\documentclass[a4paper, 12pt, reqno]{amsart}
\title{An idelic view of ideals}
\author{Shin-Eui Song}
\email{libofmath@gmail.com}

\usepackage{amsmath, amssymb, color, amsthm, extarrows, tikz, centernot, cleveref, enumitem, url}
\usepackage[margin=1.2in]{geometry}
\usetikzlibrary{matrix, arrows, positioning}

\newcommand{\red}[1]{{\color{red}#1}\index{#1}}

\newcommand{\mfk}[1]{\mathfrak{#1}}
\newcommand{\zmod}[1]{\mathbb{Z}/#1\mathbb{Z}}
\newcommand{\obar}[1]{\mkern 1.5mu\overline{\mkern-1.5mu#1\mkern-1.5mu}\mkern 1.5mu}

\newtheorem{prp}{Proposition}[section]
\newtheorem{lem}[prp]{Lemma}
\newtheorem{thm}[prp]{Theorem}
\newtheorem{cor}[prp]{Corollary}
\newtheorem{clm}[prp]{Claim}

\theoremstyle{definition}

\newtheorem*{rmk*}{Remark}

\setenumerate[1]{label=\roman*)}
\setlist{noitemsep}

\def\signed #1{{\leavevmode\unskip\nobreak\hfil\penalty50\hskip2em
  \hbox{}\nobreak\hfil(#1)%
  \parfillskip=0pt \finalhyphendemerits=0 \endgraf}}

\newsavebox\mybox
\newenvironment{aquote}[1]
  {\savebox\mybox{#1}\begin{quote}}
  {\signed{\usebox\mybox}\end{quote}}

\makeindex

\begin{document}
\begin{abstract} 
Ideles and adeles can be viewed as a generalization of Minkowski theory. Instead of embedding a number field into a product of archimedean completions, one embeds into a product of all of its completions with some restriction.
\\ \indent This survey starts with a review of point-set topology and constructs the real numbers from the rational numbers. Then we recap the concepts from local field and global fields. We, then, define adeles and ideles, and explore their properties. The ideles $\mathbb{I}_k$ modulo $k^*$ maps surjectively to the ideal class group $Cl_k$, and the compactness of $C_S^0$ will give rise to an alternative proof to the Dirichlet's $S$-unit theorem and the finiteness of ideal class group.
\end{abstract}
\maketitle
\tableofcontents
\section*{Notation}
Let $F$ be an ordered field, We use $F^{>0}:=\{ x \in F ~|~ x > 0 \}$ and similarly $F^{\geq 0}:= \{x \in F ~|~ x \geq 0 \}$. We also use $F^{<0}, F^{\leq 0}$ which are defined likewise.
\\ \indent Generally, $F$ will denote an arbitrary field and $k$ will denote a number field. $K$ will be used for a field extension of both $F$ and $k$.
\part{Preliminaries}
We assume that the reader has learned the ideal-theoretic approach. That is, the reader should be familiar with Dedekind domains, discrete valuation ring, and the unique factorization into product of prime ideals. Chapter 1 from Neukirch should suffice.
\section{Point-set Topology}~
Let $X$ be a topological space, then a \red{neighborhood} of $x \in X$ is a subset $V 
\subset X$ which contains an open set $U$ containing $x$, i.e. 
\[
	x \in U \subset V
\]
\indent The sequence $\{x_n \}$ in $X$ is said to \red{converge} to $x \in X$ if for all open set $U \subset X$ containing $x$, there exists $N \in \mathbb{N}$ such that $x_n \in U$  whenever $n \geq N$. A topological space $X$ is a \red{Hausdorff space} if for any two distinct points $x_1$ and $x_2$ in $X$, there exist two open sets $U_1$ and $U_2$ that contain $x_1$ and $x_2$ respectively, with $U_1 \cap U_2 = \varnothing$.
\begin{prp} Let $X$ be a Hausdorff space and let $\{ x_n \}$ be a convergent sequence in $X$, then $\{ x_n \}$ converges to an unique limit point.
\end{prp}
\begin{proof}
Suppose that the sequence $\{x_n\}$ converges to $x$ and $x'$, then there exists two disjoint open sets $U_x$ and $U_{x'}$ containing $x$ and $x'$ respectively. However, there exist $N_1, N_2 \in \mathbb{N}$ such that
    \[
    	\begin{array}{lcl}
        	n \geq N_1 & \Rightarrow & x_n \in U_{x} \\
            n \geq N_2 & \Rightarrow & x_n \in U_{x'}
            
        \end{array}
    \]
    then $x_n \in U_{x} \cap U_{x'}$ for $n \geq \max(N_1, N_2)$, which contradicts that the two open sets are disjoint. Hence any convergent sequence converges to an unique point in a Hausdorff space. 
\end{proof}
A point $x$ in a topological space $X$ is called an \red{isolated point} of $S \subset X$ if there exists a neighborhood $V$ of $x$ with $V \cap S = \{ x \}$. A \red{discrete set} is a subset of $X$ consisting only of isolated points. Notice that if a topological space $X$ has a discrete topology, one point subset $\{ x \}$ of $X$ is open. This is equivalent to saying that every point $x \in X$ is an isolated point of $X$. Hence we have that a subset $S \subset X$ is discrete if and only if $S$ has a discrete topology as its subspace topology.
\\ \indent A \red{limit point} $x$ of a set $Y$ is a point where every neighborhood of $x$ contains a point that is not $x$. A subset $Y$ of a topological space $X$ is called \red{dense} if either $x \in Y$ or $x$ is a limit point in $Y$. 
\begin{prp} Let $Y \subset X$ be dense and $U \subset X$ be an open set. $Y \cap U \neq \varnothing$.
\end{prp}
\begin{proof} Pick $x \in U$. If $x \notin Y$, then $x$ is a limit point of $Y$. Clearly $U$ is a neighborhood of $x$, so, by the definition of a limit point, $U \cap Y \neq \varnothing$.
\end{proof}
\begin{prp}
Let $Z$ be a Hausdorff space and $f_1, f_2 : X \rightarrow Z$ be continuous maps. If $f_1$ and $f_2$ coincide in some dense subset $Y$ of $X$, then $f_1 = f_2$. 
\end{prp}
\begin{proof}
Suppose that $f_1(x) \neq f_2(x)$ for some $x \in X$, then there exists open sets $U_1$ and $U_2$ of $Z$ containing $f_1(x)$ and $f_2(x)$ respectively that is disjoint. Define $U = f_1^{-1}(U_1) \cap f_2^{-1}(U_2)$ which is open by continuity. We have that $x \in U$, so $U$ is not empty. Also, for any $x' \in U$, 
\[
	f_1(x') \in U_1 \text{ and } f_2(x') \in U_2 \Rightarrow f_1(x') \neq f_2(x')
\]
as $U_1 \cap U_2$ is empty. Thus $f_1$ and $f_2$ disagrees in all $U$, however, because $U \cap Y \neq \varnothing$, this leads to a contradiction.
\end{proof}
Let $(X, \tau)$ be a topological space and $\sim$ be an equivalence relation on $X$. Then we let $Y = X/\sim$ and give the topology on $Y$ defined by
\[
	\tau_Y = \{ U \subset Y ~|~ p^{-1}(U) \in \tau \}
\]
where $p: X \rightarrow Y$ is the canonical epimorphism. Or equivalently, 
%p stands for projection
\[
	\tau_Y = \{ U \subset Y ~|~ \bigcup_{[x] \in U} [x] \in \tau \}
\]
The space $(Y, \tau_Y)$ will be called the \red{quotient space} of $X$ and $\tau_Y$ will be called the \red{quotient topology}.
\\ \indent Let $X_i$ be a topological space and let
\[
	X = \prod_{i \in I} X_i
\]
for some indexing set $I$ with canonical projections $p_i:X \rightarrow X_i$. The \red{product topology} of $X$ is defined to be the coarsest topology for which all $p_i$ are continuous. This means that open sets of the form 
\[
	p_i^{-1}(U_i) = U_i \times \prod_{j \in I, j \neq i} X_j
\]
must be in the product topology. As a topology is closed under union and finite intersection, we can see that all open sets are unions of the form 
\[
	\prod_{i \in S} U_i \times \prod_{i \notin S} X_i
\]
where $S \subset I$ is a finite subset. The finiteness of $S$ comes from the fact that the topology is only closed under finite intersection.
\\ \indent Let $X$ be a topological space and $K \subset X$ be a subset of $X$.  Then $\{ U_\alpha \} \subset \tau$ is called an \red{open cover} for $K$ if 
\[
	K \subset \bigcup U_\alpha
\]
and we say that $\{ U_\alpha \}$ \red{covers} $K$. An \red{open subcover} of $\{ U_\alpha \}$ is simply a subset that covers $K$. The space $X$ is called \red{compact} if for any collection of open sets $\{ U_\alpha \}$ which covers $X$, i.e.
\[
	X = \bigcup U_\alpha
\]
there exists a finite collection $U_1, \ldots, U_n \in \{ U_\alpha \}$. Similarly, a subset $K \subset X$ is called \red{compact} if for any open cover for $K$, there exists a finite open subcover for $K$. Compactness is a global property in the sense that compact set is compact space in its own right. We give the following characterization.
\begin{prp}  $K$ is a compact subset of $X$ if and only if $K$ is a compact space as a subspace of $X$.
\end{prp}
\begin{proof}
	Let $K \subset X$ be a compact subset, and let $\{ U'_\alpha \}$ be an open cover for $K$ in the subspace topology. we have $U'_\alpha = U_\alpha \cap K$ for some open set $U_\alpha$ in $X$. So we see that $\{ U_\alpha \}$ is an open cover of $K$ in $X$. There exists a finite subcover $\{ U_i \}_{i=1}^n$ by compactness. $\{ U_i \cap K \}_{i=1}^n$ gives the desired finite subcover of $\{ U'_i \}_{i=1}^n$.
    \\ \indent Conversely, let $K$ be a compact subspace of $X$.  If $\{ U_\alpha \}$ is an open cover for $K$ with $U_\alpha \subset X$ open, $\{ U_\alpha \cap K \}$ is an open cover in the subspace topology and induces a finite subcover $\{ U_i \cap K \}_{i=1}^n$ in the subspace topology. $\{ U_i \}_{i=1}^n$ is a finite subcover of $\{ U_\alpha \}$ for $K$.
\end{proof}
\begin{prp} Let $S$ be a closed subset of a compact set $K$, then $S$ is also compact.
\end{prp}
\begin{proof} Let $\{ U_ \alpha \}$ be an open cover for $S$, then  $\{ U_ \alpha \} \cup \{ S^c \}$ is an open cover for $K$. Then there exists $\{ U_i \}_{i=1}^n \cup \{ S^c \}$ a finite subcover. Because $S^c \cap S = \varnothing$, $\{ U_i \}_{i=1}^n$ gives a subcover for $S$.
\end{proof}
\begin{prp} Let $X$ be a compact space, then all its quotient space is also compact.
\end{prp}
\begin{proof} Let $Y$ be a quotient space of $X$, and let $\{ U_\alpha \}$ be a open cover for $Y$, $\{ q^{-1}(U_\alpha) \}$ is an open cover of $X$. Hence there exists a finite subcover $\{ q^{-1}(U_i) \}_{i=1}^n$. So, by surjectivity, $\{ q(q^{-1}(U_i)) \}_{i=1}^n = \{ U_i \}_{i=1}^n$ is a finite subcover of $\{ U_\alpha \}$.
\end{proof}
\begin{prp} Let $X$ be a discrete and compact space, then $X$ has only finitely many points.
\end{prp}
\begin{proof} If $X$ is endowed with a discrete topology, then collection of all singletons $\{ x \}$ for $x \in X$ is an open cover of $X$. As there exists a finite subcover, we immediately get $X$ is finite.
\end{proof}
\begin{thm}[Tychonoff] Let $\{ K_\alpha \}$ be any collection of compact spaces $K_\alpha$, then 
\[
	\prod K_\alpha
\]
is a compact space.
\end{thm}
Let $X$ be a topological space, then $X$ is called \red{locally compact} if every point of $X$ has a compact neighborhood. A compact space $X$ is also locally compact, because for all point $x \in X$, the space $X$ becomes the compact neighborhood of $x$.
\begin{prp} Let $X_1, \ldots, X_n$ be locally compact spaces, then 
\[
	X = \prod_{i=1}^n X_i
\]
is a locally compact space.
\end{prp}
\begin{proof} Let $(x_i)_{i=1}^n$ be a point in $X$, then for each $x_i$, there exists a compact neighborhood $V_i$ and an open set $U_i$ containing $x_i$ such that $U_i \subset V_i$. Then
\[
	(x_i) \in \prod U_i \subset \prod V_i
\]
then $\prod V_i$ is compact by Tychonoff's theorem and $\prod U_i$ is open in the product topology as it is a finite product. Hence $(x_i)$ has a compact neighborhood.
\end{proof}
Infinite product of locally compact spaces is not locally compact in general. The above proof goes wrong as infinite product $\prod U_i$ of open sets are not in general open. By \cite{TaoBlog}, we have that the only locally compact Hausdorff topological vector space are finite-dimensional. Infinite product of Hausdorff topological space over real cannot be finite-dimensional, so we see that it will never be locally compact. To resolve this problem, we will later introduce the restricted product. 
\\ \indent Let $G$ be a group which is endowed with a topology. If the multiplication $* : G \times G \rightarrow G$ is continuous with $G \times G$ given a product topology and the inversion ${\color{white}x}^{-1} : G \rightarrow G$ is continuous, then $G$ is called a \red{topological group}. If $R$ is a ring endowed with a topology such that the addition $+: R \times R \rightarrow R$, negation $-:R \rightarrow R$, and multiplication $*:R \times R \rightarrow R$ are continuous then $R$ is called a \red{topological ring}. Finally, a field $F$ is called a \red{topological field} if $F$ is a topological ring and, in addition, ${\color{white}x}^{-1}: F^\times \rightarrow F^\times$ is continuous. In short, a topological algebraic structure is a topological space if all its operations are continuous.
\\ \indent We now give an example of a locally compact topological field $\mathbb{R}$ constructed from $\mathbb{Q}$ and will prove the basic properties of $\mathbb{R}$.
\section{Construction of $\mathbb{R}$}
One of the famous methods in completing the rational numbers $\mathbb{Q}$ is to use Dedekind cuts, but in this survey we will use Cauchy sequences to fill in the gaps in the rational numbers. The following procedure generalizes to number fields and will be use to construct ideles and adeles. Rational numbers will be denoted $a,b$ to be distinguished with \emph{real numbers} soon to be defined. The \red{absolute value} $|{\color{white}x}| : \mathbb{Q} \rightarrow \mathbb{Q}^{\geq 0}$ is defined by
\[
	|a| := \left \{ \begin{array}{lcl} a & \text{ if } & a \geq 0 \\
    								   -a & \text{ if } & a < 0
                    \end{array}
                    \right .
\]
The absolute value satisfies the following well-known properties
\begin{prp} Let $|{\color{white}x}|$ be the absolute value of $\mathbb{Q}$ defined above, then
	\begin{enumerate}
    	\item $|a| \geq 0$ and $|a|=0$ if and only if $a=0$,
        \item $|ab| = |a||b|$,
        \item $|a+b| \leq |a| + |b|$ (Triangle inequality)
    \end{enumerate} \label{absprp}
\end{prp}
We omit the tedious proof. A sequence $\{ a_n \}_{n=1}^\infty$ or simply $\{ a_n \}$ in $\mathbb{Q}$ is called a \red{Cauchy sequence} if for any $\varepsilon \in \mathbb{Q}^{> 0}$, there exists a positive integer $N \in \mathbb{N}$ satisfying
\[
	|a_n- a_m| < \varepsilon \text{ whenever } n, m \geq N
\]
The terms in a Cauchy sequence become closer together as the indices increase; however, there are Cauchy sequences that do not converge in a \emph{non-complete} spaces. The construction of $\mathbb{R}$ starts by constructing the ring of Cauchy sequence and embedding the rationals to the ring of Cauchy sequence of $\mathbb{Q}$. Then we consider the quotient ring which turns out to field by taking the quotient of some \emph{nice} sequences. 
\begin{lem} Every Cauchy sequence $\{ a_n \}$ is bounded, i.e. there exists $M \in \mathbb{Q}^{\geq 0}$ such that $|a_n| \leq M$ for all $n$.
\end{lem}
\begin{proof} Let $\varepsilon = 1$, then there exists $N$ such that $n \geq N$ implies that $|a_n-a_N| < \varepsilon$. We let $M = \max \{ |a_n| ~|~ n \leq N \} + 1$, then 
\[
	\begin{array}{lclc}
    	|a_n| & \leq & M & \text{ if } n < N \\
        |a_n| & \leq & |a_n - a_N| + |a_N| & \\ 
        & < & 1 + (M-1) & \\
        & = & M & \text { if } n \geq N
    \end{array}
\]
\end{proof}
\begin{prp} \label{cauchyring} Let $\{ a_n \}$ and $\{ b_n \}$ be Cauchy sequences, then so are $\{ a_n + b_n \}$ and $\{ a_n b_n\}$.
\end{prp}
\begin{proof}
	For a fixed $\varepsilon \in \mathbb{Q}^{> 0}$, there exists $N_1, N_2$ such that
    \[
    	\begin{array}{ll}
        	|a_n - a_m| < \varepsilon/2 & \forall~n,m \geq N_1 \\
            |b_n - b_m| < \varepsilon/2 & \forall~n,m \geq N_2
        \end{array}
    \]
then if we let $N= \max(N_1, N_2)$, then 
\[
	|(a_n + b_n) - (b_m + b_m)| \leq |a_n - a_m| + |b_n - b_m| < \varepsilon/2 + \varepsilon/2 = \varepsilon
\]
hence sum of Cauchy sequences is a Cauchy sequence.
\\ \indent For the product, we have that $|a_n|, |b_n| \leq B$ for some bound $B$, and for a fixed $\varepsilon \in \mathbb{Q}^{> 0}$, there exists $N$ such that $|a_n -a_m|, |b_n - b_m | < \varepsilon/2B$ for all $n,m \geq N$. Then we get
\[
	\begin{array}{lcl}
    	|a_n b_n - a_m b_m| & = & |a_n b_n - a_n b_m +a_n b_m - a_m b_m | \\
        & \leq & |a_n| |b_n - b_m| + |b_m| |a_n - a_m | \\
        & \leq & B \cdot \dfrac{\varepsilon}{2B} + B \cdot \dfrac{\varepsilon}{2B} = \varepsilon
    \end{array}
\]
hence product of Cauchy sequences is a Cauchy sequence.
\end{proof}
The set of all Cauchy sequences $\mathcal{R}$ is a ring. In fact, $\mathcal{R}$ forms a subring of 
\[
	\prod_{i=1}^\infty \mathbb{Q}
\]
the countably infinite product of $\mathbb{Q}$. Let $A = \{ a_n \} \in \mathcal{R}$ such that
\[
	a_n = \left \{ \begin{array}{ll}
    	1 & \text{ if } n=1 \\
        0 & \text{ if } n \neq 1
    \end{array} \right .
\]
and $B = \{ b_n \} \in \mathcal{R}$ such that
\[
	b_n = \left \{ \begin{array}{ll}
    	1 & \text{ if } n=2 \\
        0 & \text{ if } n \neq 2
    \end{array} \right .
\]
$AB = 0 \in \mathcal{R}$, hence we see that $\mathcal{R}$ is not a domain.  
\\ \indent Let $\mathcal{N} = \{ A \in \mathcal{R} ~|~ a_n \rightarrow 0 \text{ as } n \rightarrow \infty \}$ be the set of all \red{null sequences}, a Cauchy sequence that converges to 0. 
\begin{prp} $\mathcal{N}$ is an ideal of $\mathcal{R}$
\end{prp}
\begin{proof} Let $B, C \in \mathcal{N}$, and $A \in \mathcal{R}$, with $B= \{ b_n \}$ and $C = \{ c_n \}$, then for any given $\varepsilon \in \mathbb{Q}^{> 0}$, there exist $N \in \mathbb{N}$ such that $n \geq N$ implies that
\[
	\begin{array}{lcl}
    	|b_n| & < & \varepsilon/2 \\ 
        |c_n| & < & \varepsilon/2
    \end{array}
\]
hence $|b_n + c_n| \leq |b_n| + |c_n| < \varepsilon/2 + \varepsilon/2 = \varepsilon$. Hence $B+C \in \mathcal{N}$. 
\\ \indent Also, let $A \in \mathcal{R}$, there exists $M \in \mathbb{Q}^{\geq 0}$ which bounds $|a_n|$, and there exists $N \in \mathbb{N}$ such that $|b_n| < \varepsilon/M$ whenever $n \geq N$, then for $n \geq N$,
\[
	\begin{array}{lcl}
    	|a_nb_n | & = & |a_n||b_n| \\
        & \leq & M |b_n| \\
        & < & \varepsilon
    \end{array}
\]
This shows that $AB \in \mathcal{N}$, so $\mathcal{N}$ is an ideal of $\mathcal{R}$.
\end{proof}
Let $\mathbb{R} = \mathcal{R}/\mathcal{N}$, then we would like to discuss the nonzero element in $\mathbb{R}$. We will use $x,y \in \mathbb{R}$ for the elements and $x = [A]$ will denote the equivalence relation of $A \in \mathcal{R}$. If $A=\{ a_n \}$, then we agree that $[a_n] := [A]$. Let $[A] \in \mathbb{R}$ be nonzero, then $A$ is a Cauchy sequence that does not converge to $0$. By taking the converse of the definition of convergence, we have 
\begin{center}
There exists $\varepsilon \in \mathbb{Q}^{> 0}$ such that for all $N \in \mathbb{N}$, there exists $M \geq N$ such that
$|a_M| \geq \varepsilon$. ($\dagger$)
\end{center} 
\begin{prp} \label{inversenull} The composition
\[
	\mathbb{Q} \hookrightarrow \mathcal{R} \twoheadrightarrow \mathbb{R}
\]
is an inclusion and $\mathbb{R}$ is a field.
\end{prp}
\begin{proof} Let $[a_n] \in \mathbb{R}$ nonzero. Fix $\varepsilon \in \mathbb{Q}^{> 0}$, then $\{ a_n \}$ Cauchy implies that there exists $N$ such that $|a_n - a_m| < \varepsilon$ whenever $n,m \geq N$. By ($\dagger$), there exists $M \geq N$ with $|a_M| > \varepsilon$. If we let $n \geq M$, then 
\[
	\bigl| |a_n| - |a_M| \bigr| \leq |a_n - a_M| < \varepsilon
\]
implies that $- \varepsilon < |a_n| - |a_M| \Rightarrow 0 < |a_M| - \varepsilon < |a_n|$. ($\dagger \dagger$) Hence there exists $N$ such that $|a_n| > 0$ for all $n \geq N$.
\\ \indent Let $[b_n]$ with $b_n = 0$ for $n < M$ and $b_n = a_n^{-1}$ for $n \geq M$. Then we get $[a_n][b_n] = [1]$. But we must first prove that $\{ b_n \}$ is a Cauchy sequence. By the notation above, for all $n \geq M$, we get
\[
	|a_M| - \varepsilon < |a_n| \Rightarrow \dfrac{1}{|a_n|} < \dfrac{1}{|a_M|- \varepsilon}
\]
Let $B = 1/(|a_M| - \varepsilon)$, then for a fixed $\gamma \in \mathbb{Q}^{> 0}$, there exists $N \geq M$, such that
\[
	|a_n - a_m| < \gamma/B^2 \text{ whenever } n,m \geq N
\]
which implies
\[
	|a_n^{-1} - a_m^{-1}| = \dfrac{|a_n - a_m|}{|a_n||a_m|} < \dfrac{\gamma}{B^2} \cdot B^2 = \gamma
\]
This shows that $\{ b_n \}$ is a Cauchy sequence. We have shown that every nonzero element in $\mathbb{R}$ has a multiplicative inverse, hence $\mathbb{R}$ is a field once we prove that $[1] \neq [0]$. But clearly $\{1\}$ does not converge to $0$. Finally for any $a,b \in \mathbb{Q}$, 
\[
	b \neq a \Rightarrow a-b \neq 0 \Rightarrow \{ a-b \} \not\rightarrow 0 \Rightarrow [x-y] \neq 0   
\]
proves the injectivity of the composition.
\end{proof}
We have extended $\mathbb{Q}$ into a larger field $\mathbb{R}$ using the absolute value $|{\color{white}x}|$. The absolute value can be also extended canonically to a function of $\mathbb{R}$ as follows,
\[
	|{\color{white}x}|: \mathbb{R} \rightarrow \mathbb{R}
\]
defined by 
\[
	[a_n] \mapsto [|a_n|]
\]
The inverse triangle inequality
\[
	\bigl| |a_n| - |a_m| \bigr| \leq |a_n-a_m|
\]
shows that $\{ |a_n| \}$ is a Cauchy sequence. Also let $\{a_n\}$ be a null sequence, then $[|a_n|] = [0]$ as $||a_n|| = |a_n|$, so the map is well-defined.We want to check that the absolute value extended to $\mathbb{R}$ satisfies the three properties stated in \cref{absprp}. However, we have not defined an order in $\mathbb{R}$. Suppose $x = [a_n] \in \mathbb{R}$ nonzero, then we know that $|a_n| - a_n = 0$ or $-2a_n$ and that there exists $M$ such that $B \leq |a_n|$ for some $B \in \mathbb{Q}^{\geq 0}$ and for $n \geq M$ by ($\dagger \dagger$) in \Cref{inversenull}. Let $\varepsilon < 2B$, then there exists $N \geq M$ such that
\[
	 \bigl| (|a_n|-a_n)  - (|a_m| - a_m) \bigr| < \varepsilon \text{ whenever } n, m \geq N
\]
as $\{ |a_n| -a_n \}$ is also a Cauchy sequence. If $|a_n|-a_n = 0$ and $|a_m| - a_m = -2a_m$, then we get
\[
    2B \leq |-2|a_m|| < \varepsilon < 2B
\] 
leads to a contradiction. Hence this shows that $n \geq N$, $|a_n| - a_n = 0$ or $|a_n| - a_n = -2a_n$. In particular, $|x| = x$ or $|x|=-x$ for $x \in \mathbb{R}$. Thus we have that the absolute value of $\mathbb{R}$ is analogous to the case of $\mathbb{Q}$. We denote
\[
	\mathbb{R}^{\geq 0}:= \{ x \in \mathbb{R} ~|~ |x| = x \}
\]
and it isn't hard to see that $\mathbb{Q}^{\geq 0} \subset \mathbb{R}^{\geq 0}$. We define $\mathbb{R}^{>0}$ similarly, and define elements in $\mathbb{R}^{>0}$ to be \red{positive} and elements in $\mathbb{R}^{<0}$ to be \red{negative}. From now on, when we say the absolute value of $\mathbb{R}$, then we mean $|{\color{white}x}| : \mathbb{R} \rightarrow \mathbb{R}^{\geq 0}$ such that
\[
	|x| = \left \{ \begin{array}{lcl}
    	x & \text{ if } & x \in \mathbb{R}^{\geq 0} \\
        -x & \text{ if } & x \in \mathbb{R}^{< 0}
    	\end{array} \right .
\]
Now we say $x \leq y$ if $y-x \in \mathbb{R}^{\geq 0}$. Observe that if $x=[a_n]$ and $y=[b_n]$, then $x \leq y$ is equivalent to saying that there exists $N$ such that whenever $n \geq N$,
\[
    a_n - b_n \leq 0
\]
With this order, we can prove that the new absolute value $|{\color{white}x}|$ satisfies the three properties of an absolute value, and $d(x,y):=|x-y|$ defines a metric on $\mathbb{R}$. Cauchy sequence and convergence of $\mathbb{R}$ is analogous to the case of $\mathbb{Q}$.
\begin{prp} $\mathbb{Q}$ is a dense subfield of $\mathbb{R}$.
\end{prp}
\begin{proof} This proposition explains the motivation of the construction of $\mathbb{R}$ using Cauchy sequences. Given $x = [a_n] \in \mathbb{R}$ and $\varepsilon \in \mathbb{Q}^{> 0}$, there exists $N$ such that
\[
    |a_n - a_m| < \varepsilon \text{ whenever } n, m \geq N
\]
In particular,
\[
	|a_n - a_N| < \varepsilon \text{ whenever } n \geq N
\]
Exactly says that
\[
    |x - [a_N]| < [\varepsilon] \text{ in } \mathbb{R}
\]
This shows that $\mathbb{Q}$ is dense.
\end{proof}
The proposition above shows that the statements of the form ``For a fixed $\varepsilon \in \mathbb{Q}^{> 0}$'' can be replaced by ``For a fixed $\varepsilon \in \mathbb{R}^{> 0}$'' by denseness. A field is \red{Dedekind complete} if it satisfies the least upper bound property and is \red{Cauchy complete} if all Cauchy sequence converges. We will first prove that $\mathbb{R}$ is Cauchy complete and prove that it is Dedekind complete.
\begin{thm} $\mathbb{R}$ is Cauchy complete.
\end{thm}
\begin{proof} Let $\{ a_n \} \in \mathcal{R}$, then we have $\{ [a_n] \}$ a Cauchy sequence in $\mathbb{R}$. The definition of Cauchy sequence immediately implies that $\{ [a_n] \} \rightarrow [a_n]$. Hence all rational Cauchy sequence converges in $\mathbb{R}$. Let $\{ x_n \}$ be a Cauchy sequence in $\mathbb{R}$, then there exists $r_n \in \mathbb{Q}$ such that 
\[
	|x_n - r_n| < 1/n
\]
by denseness, hence $\lim_{n \rightarrow \infty} x_n - r_n = 0$. The convergence of $\{ r_n \}$ implies the convergence of $\{ x_n \}$.
\end{proof}
As the purpose of this survey is not introducing analysis, we assume properties such as the least upper bound property, archimedean properties, and Heine-Borel, and other standard results from analysis.
% \begin{thm} $\mathbb{R}$ satisfies the least upper bound property.
% \end{thm}
% \begin{prp}[Archimedean Property] Let $x \in \mathbb{R}^\times$, then for any $y \in \mathbb{R}$, there exists an integer $n$ such that
% \[
% 	nx > y
% \]
% \end{prp}
% \begin{prp} Let $\alpha \in \mathbb{R}^{>0}$, then for all $x,y \in \mathbb{R}$,  $x < y$ implies
% \[
% 	x^\alpha < y^\alpha
% \]
% \end{prp}
\begin{thm} \label{Rlocallycompact} $\mathbb{R}$ is a locally compact topological space.
\end{thm}
\begin{proof} By Heine-Borel theorem, $x \in B(x, 1) \subset \overline{B(x,1)}$, hence $\mathbb{R}$ is a locally compact space. So we prove that the four operations $+, -, *, {\color{white}x}^{-1}$ is continuous.
\\ \indent Let $f = +$, and $U = B(g,r)$. We fix $(a,b) \in f^{-1}(U)$ which shows that $|g-a-b| =r' < r$. There exists $\varepsilon > 0$ such that $r'+\varepsilon < r$. Now consider two open sets
\[
	U_1 = B(a,\varepsilon/2) \text{ and } U_2 = B(b, \varepsilon/2)
\]
Let $(c,d) \in U_1 \times U_2$, then 
\[
	\begin{array}{lcl}
    	|g-c-d| & = & |g-a-b +a + b -c -d| \\
        & \leq & |g-a-b| +|a+b-c-d| \\
        & < & r' + \varepsilon/2 + \varepsilon/2 =  r' + \varepsilon < r
    \end{array}
\]
shows that $(a,b) \in U_1 \times U_2 \subset f^{-1}(U)$. $f^{-1}(U)$ is open and $+$ is continuous.
\\ \indent Let $f = -$, then for $B(x,r)$,
\[
	\begin{array}{lcl}
    	f^{-1}(B(x,r)) & = & \{ y \in \mathbb{R} ~|~ -y \in B(x,r) \} \\
        & = & \{ y \in \mathbb{R} ~|~ |-y -x| < r \} \\
        & = & \{ y \in \mathbb{R} ~|~ |(-x) - y| <r \} \\
        & = & B(-x,r)
    \end{array}
\]
hence the preimage is open, and $-$ is continuous.
\\ \indent Let $f=*$ and $U=B(g,r)$ and fix $(a,b) \in f^{-1}(U)$. Let $|g-ab| = r'$ and we have $\varepsilon > 0$ such that $r'+\varepsilon < r$. We let $\varepsilon_1 = \min( \sqrt{\varepsilon}/\sqrt{3}, \varepsilon/3|b|)$ and $\varepsilon_2 = \min(\sqrt{\varepsilon}/\sqrt{3}, \varepsilon/3|a|)$. Consider two open sets
\[
	U_1 = B(a, \varepsilon_1) \text{ and } U_2 = B(b, \varepsilon_2)
\]
then for all $(c,d) \in U_1 \times U_2$, we have $c = a+ \gamma_1$ and $d= b+ \gamma_2$ with $|\gamma_i| < \varepsilon_i$. We then have
\[
	\begin{array}{lcl}
    	|ab-cd| & = & |ab-(a+\gamma_1)(b+\gamma_2)| \\
        & = & |-a\gamma_2 -b\gamma_1 - \gamma_1\gamma_2| \\
        & \leq & |a\gamma_2| + |b\gamma_1| + |\gamma_1 \gamma_2| \\
        & < & |a| \dfrac{\varepsilon}{3|a|} + |b| \dfrac{\varepsilon}{3|b|} + \dfrac{\sqrt{\varepsilon}}{\sqrt{3}} \dfrac{\sqrt{\varepsilon}}{\sqrt{3}} \\ 
        & = & \dfrac{\varepsilon}{3} + \dfrac{\varepsilon}{3} + \dfrac{\varepsilon}{3} = \varepsilon.
    \end{array}
\]
We conclude by observing that
\[
	|g-cd| = |g-ab+ab-cd| \leq |g-ab| + |ab-cd| < r' + \varepsilon < r
\]
which proves that $f^{-1}(U)$ is open.
\\ \indent Let $f = {\color{white}~}^{-1}: \mathbb{R}^* \rightarrow \mathbb{R}^*$ and $U=B(x,r)$, then fix $y \in f^{-1}(U)$, the consider
\[
	\gamma = \min \left ( \dfrac{|y|}{2}, \dfrac{|y|^2}{2} \varepsilon \right )
\]
with open set $B(y, \gamma)$. Similarly from above, we let $|x-y^{-1}| = r'$ and let $\varepsilon \in \mathbb{R}^{>0}$ such that $r'+\varepsilon < r$. Fix $z \in B(y, \gamma)$, then we have that
\[
	\bigl | |z|-|y| \bigr | \leq |z-y| < \dfrac{|y|}{2}
\]
implies that 
\[
	|y| - |z| < \dfrac{y}{2} \Rightarrow \dfrac{1}{2} |y| < |z| \Rightarrow \dfrac{1}{|z|} < \dfrac{2}{|y|}
\] 
The above inequality not gives us
\[
	\begin{array}{lcl}
    	|x-z^{-1}| & = & | x-y^{-1} + y^{-1} - z^{-1}| \\
        & \leq & |x - y^{-1}| + \dfrac{|y-z|}{|y||z|} \\
        &  < & r' + \left ( \dfrac{2}{|y|} \right ) \left ( \dfrac{|y|^2}{2} \varepsilon \right ) \left ( \dfrac{1}{|y|} \right )  \\
        & = & r' + \varepsilon < r
    \end{array}
\]
\end{proof}
%%%%%%%%%%%%%%%%%%%%%%%%%%%%%%%%%%%%% new section %%%%%%%%%%%%%%%%%%%%%%%%
\newpage
\part{Local Fields}
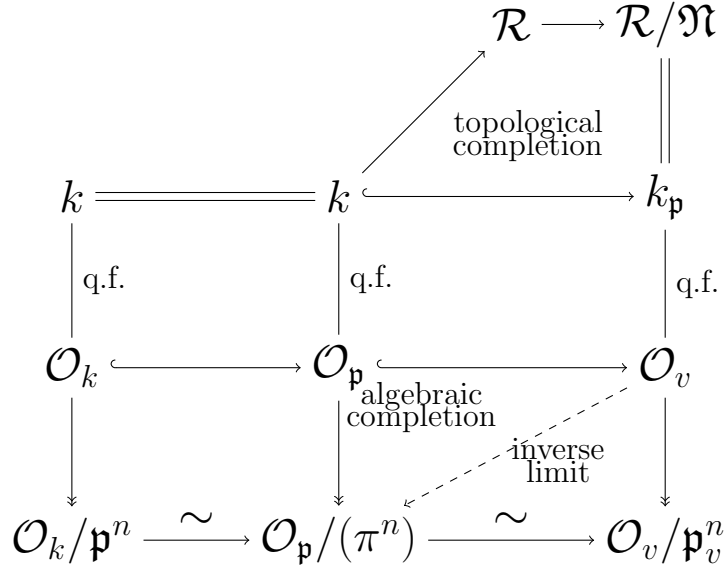
\begin{center}
\begin{tikzpicture}
	\tikzstyle{every node}=[font=\Large]
	\matrix (m)[matrix of math nodes, nodes={anchor=center}, row sep=40pt, column sep=20pt]{
    	& & & \mathcal{R} & \mathcal{R}/\mfk{N} \\
        k & & k & & k_\mfk{p} \\
        \mathcal{O}_k & & \mathcal{O}_\mfk{p} & & \mathcal{O}_v \\
        \mathcal{O}_k/\mfk{p}^n & & \mathcal{O}_\mfk{p}/(\pi^n) & & \mathcal{O}_v/\mfk{p}_v^n \\
    };
    	\draw[->]
        	(m-2-3) edge (m-1-4)
            (m-1-4) edge (m-1-5)
            (m-4-1) edge node[above]{$\sim$} (m-4-3)
            (m-4-3) edge node[above]{$\sim$} (m-4-5);
        \draw[transform canvas={yshift=1.5pt}]
            (m-2-1) edge (m-2-3);
        \draw[transform canvas={yshift=-1.5pt}]
            (m-2-1) edge (m-2-3);
        \draw[transform canvas={xshift=1.5pt}]
        	(m-1-5) edge (m-2-5);
        \draw[transform canvas={xshift=-1.5pt}]
        	(m-1-5) edge (m-2-5);    
        \draw[right hook->]
        	(m-3-1) edge (m-3-3)
            (m-3-3) edge (m-3-5)
            (m-2-3) edge (m-2-5);
        \draw[-]
        	(m-2-1) edge node[right]{\normalsize q.f.} (m-3-1)
            (m-2-3) edge node[right]{\normalsize q.f.} (m-3-3)
            (m-2-5) edge node[right]{\normalsize q.f.} (m-3-5);
        \draw[->>]
        	(m-3-1) edge (m-4-1)
            (m-3-3) edge (m-4-3)
            (m-3-5) edge (m-4-5);
        \draw[dashed,->]
        	(m-3-5) edge (m-4-3);
        \node[] at (0.7, -1.5) {\normalsize algebraic};
        \node[] at (0.7, -1.8) {\normalsize completion};
        \node[] at (2.5, -2.2) {\normalsize inverse};
        \node[] at (2.5, -2.5) {\normalsize limit};
        \node[] at (2.1, 2.1) {\normalsize topological};
        \node[] at (2.1, 1.8) {\normalsize completion};
\end{tikzpicture}
\end{center}
$k$: number field
\\ $\mathcal{O}_k$: the ring of integers of $k$
\\ $\mathcal{O}_\mfk{p}$: the valuation ring of $k$ with respect to $v_\mfk{p}$ or localization $(\mathcal{O}_k)_\mfk{p}$
\\ $\mathcal{O}_v$: the inverse limit of $\mathcal{O}_\mfk{p}/(\pi^n)$, or the valuation ring of $k_\mfk{p}$ 
\\ $\mathcal{R}$: ring of cauchy sequence of $k$,
\\ $\mfk{N}$: maximal ideal of $R$ consisting of null sequences
\\ $k_\mfk{p}$: completion of $k$ with respect to $v_\mfk{p}$
\\ $\mfk{p}$: prime ideal of $\mathcal{O}_k$
\\ $\pi$: prime element of $\mathcal{O}_\mfk{p}$
\\ $\mfk{p}_v$: the unique maximal ideal of $\mathcal{O}_v$

%
%
% New Section: Inverse Limits
%
%

\section{Abstract Valuation}
A \red{valuation} of the field $F$ is a function $\varphi: F \rightarrow \mathbb{R}^{\geq 0}$ satisfying
\begin{enumerate}
	\item $\varphi(x)=0 \Leftrightarrow x = 0$,
    \item $\varphi(xy) = \varphi(x) \varphi(y)$,
    \item There exists $C$ such that for any $x \in F$,
    \begin{equation} \label{valdef3}
    	\varphi(x) \leq 1 \Rightarrow \varphi(1+x) \leq C
	\end{equation}
\end{enumerate}
Let $\varphi$ be a valuation such that $\varphi(x)=0$ for $x=0$ and $\varphi(x) = 1$ for all $x \in F^*$, then $\varphi$ is called a \red{trivial valuation}. 
\\ \indent Any valuation $\varphi$ defines a homomorphism of $F^* \rightarrow \mathbb{R}^{>0}$. We have that $\varphi(1) = \varphi(1 \cdot 1) = \varphi(1) \varphi(1) \Rightarrow \varphi(1) = 1$. Also, $1= \varphi(x x^{-1}) = \varphi(x) \varphi(x^{-1})$. Dividing by $\varphi(x)$, $\varphi(x^{-1}) = \varphi(x)^{-1}$. Furthermore, $1= \varphi(-1)\varphi(-1) = \varphi(-1)^2$, then $\varphi(-1) = 1$. Any root of unity in $\mathbb{C}$ is of the form
\[
	\cos(2\pi m/n) + i \sin(2\pi m/n)
\]
where $1 \leq m \leq n$. The number of the form above is real if and only if $\sin(2 \pi m/n) = 0$ if and only if $\cos(2 \pi m/n) = \pm 1$. Therefore, if $\zeta$ is a root of unity in $F$, $\varphi(\zeta) = 1$. 
\\ \indent Let $F$ be a finite field, then for nonzero $a \in F$, there exists $n \neq m$ positive integers such that $x^n = x^m$. However, $\varphi(x) >0$ with $\varphi(x)^n \neq \varphi(a)^m$ whenever $n \neq m$. Hence, only valuation of a finite field $F$ is the trivial valuation.
\begin{prp} (\ref{valdef3}) can be replaced by
\begin{equation}
	\varphi(x+y) \leq C \max(\varphi(x), \varphi(y))
\end{equation}
\end{prp}
\begin{proof}
	Without loss of generosity, assume that $\varphi(x) \leq \varphi(y)$, then $\varphi(x/y) \leq 1 \Rightarrow \varphi(x/y+1) \leq C$. Multiplying $\varphi(y)$ on both side we get
    \[
    	\varphi(x+y) \leq C\varphi(y) = C\max(\varphi(x), \varphi(y))
    \]
    The case when $\varphi(y) \leq \varphi(x)$ is exactly the same.
    \\ \indent Conversely, $\varphi(x) \leq 1$, then $\max(\varphi(x), \varphi(1)) = \varphi(1)$, so 
    \[
    	\varphi(x+1) \leq Cmax(\varphi(x), \varphi(1)) = C
    \]
\end{proof}
\begin{lem} \label{valalpha} If $\alpha \in \mathbb{R}^{>0}$ and $\varphi$ be a valuation of $F$, then $\varphi^\alpha$ is also a valuation of $F$.
\end{lem}
\begin{proof} ~
	\begin{enumerate}
    	\item $\varphi^\alpha(x)=0 \Leftrightarrow (\varphi(x))^\alpha = 0 \Leftrightarrow \varphi(x) =0 \Leftrightarrow x=0$,
        \item $\varphi^\alpha(xy) = (\varphi(xy))^\alpha = (\varphi(x)\varphi(y))^\alpha = \varphi(x)^\alpha \varphi(y)^\alpha = \varphi^\alpha(x) \varphi^\alpha(y)$,
        \item Let $C \in \mathbb{R}$ such that
        \[
        	\varphi(x) \leq 1 \Rightarrow \varphi(x+1) \leq C
        \]
        then
        \begin{equation} \label{Salphabij}
        	\varphi^\alpha(x) \leq 1 \Rightarrow \varphi(x)^\alpha \leq 1 \Rightarrow \varphi(x) \leq 1 \Rightarrow \varphi(x+1) \leq C \Rightarrow \varphi(x+1)^\alpha \leq C^\alpha
        \end{equation}
        The last line shows that if $C$ is a real number number that satisfies (\ref{valdef3}) for $\varphi$, then $C^\alpha$ satisfies (\ref{valdef3}) for $\varphi^\alpha$.
    \end{enumerate}
\end{proof}
If $\varphi$ is a valuation of $F$, we define $S_\varphi$ to be the set of all real numbers satisfying the condition (\ref{valdef3}), i.e.
\[
	S_\varphi:= \{ C \in \mathbb{R} ~|~ \varphi(x) \leq 1 \Rightarrow \varphi(1+x) \leq C \text{ for all } x \in F\}
\]
Then $1=\varphi(1) = \varphi(1+0) \leq C$ shows that $1$ is a lower bound of $S_\varphi$. Thus we may define the \red{norm} of a valuation denoted
\[
	||\varphi|| = \inf{S_\varphi}
\]
\begin{prp} Let $C = ||\varphi||$, then $C \in S_\varphi$.
\end{prp}
\begin{proof}
	$C = ||\varphi||$ implies that there exists a sequence of numbers $D_n$ that converges to $C$. In fact, for each $n$, there exists $D_n \in S_\varphi$ such that $D_n - C < 1/n$, then $D_n \rightarrow C$. For any $\varphi(x) \leq 1$, we have that
    \[
    	\varphi(1+x) \leq D_n
    \]
then $n \rightarrow \infty$, we get
	\[
    	\varphi(1+x) \leq C
    \]
which exactly states that $C \in S_\varphi$.
\end{proof}
\begin{prp} For all $\alpha \in \mathbb{R}^{>0}$, $||\varphi^\alpha|| = ||\varphi||^\alpha$.
\end{prp}
\begin{proof} If $C \in S_\varphi$, then by (\ref{Salphabij}), $C^\alpha \in S_{\varphi^\alpha}$. Hence we get a canonical bijection between $S_\alpha$ and $S_{\varphi^\alpha}$ by
\[
	C \mapsto C^\alpha
\]
\begin{enumerate}
	\item Let $||\varphi||=C$, for all $D \in S_{\varphi^\alpha}$, $D^{1/\alpha} \in S_\varphi \Rightarrow C^\alpha \leq D$, hence $C^\alpha$ is a lower bound of $S_{\varphi^\alpha}$.
    \item $C \in S_\varphi \Rightarrow S^\alpha \in S_{\varphi^\alpha}$.
\end{enumerate}
This concludes that $||\varphi||^\alpha = \inf{S_{\varphi^\alpha}} = ||\varphi^\alpha||$.
\end{proof}
\begin{lem} Let $\varphi$ be a valuation of $F$, then 
\[
	||\varphi|| \leq 2 \Rightarrow \varphi(x_1 + \cdots + x_n) \leq 2n \max(\varphi(x_1), \ldots, \varphi(x_n))
\]
\end{lem}
\begin{proof} $||\varphi|| \leq 2$ implies that $\varphi(x_1 + x_2) \leq 2\max(\varphi(x_1), \varphi(x_2))$, so by induction, we have 
\[
	\varphi(x_1 + x_2 + \cdots +x_r) \leq 2^r \max(\varphi(x_1), \ldots, \varphi(x_r))
\]
then for any $n$, we have a unique $r$ such that $2^r \leq n < 2^{r+1}$, then
\[
	\begin{array}{lcl}
    	\varphi(x_1+ \cdots + x_n) & = & \varphi(x_1+ \cdots + x_n + 0 + \cdots + 0) \\
        & \leq & 2^{r+1}\max(\varphi(x_1), \ldots, \varphi(x_n)) \\
        & \leq & 2n \max(\varphi(x_1), \ldots, \varphi(x_n))
    \end{array}
\]
where we have added $0$ in order to make it into $2^{r+1}$ terms.
\end{proof}
\begin{prp} Let $\varphi$ be a function from $F$ to $\mathbb{R}^{\geq 0}$ satisfying
\begin{enumerate}
	\item $\varphi(x) = 0 \Leftrightarrow x = 0$,
    \item $\varphi(xy) = \varphi(x)\varphi(y)$,
\end{enumerate}
then we have that
\[
	||\varphi|| \leq 2 \Leftrightarrow \varphi(x+y) \leq \varphi(x) + \varphi(y)
\]
\end{prp}
\begin{proof} If the triangle equality holds,
\[
	\varphi(x+y) \leq \varphi(x) + \varphi(y) \leq \max(\varphi(x), \varphi(y)) + \max(\varphi(x), \varphi(y)) = 2 \max(\varphi(x), \varphi(y))
\]
hence $||\varphi|| \leq 2$.
\\ \indent Suppose $||\varphi|| \leq 2$, then we observe
\[
	\begin{array}{lcl}
    	(\varphi(x+y))^n & = & \displaystyle \varphi \left ( x^n + \binom{n}{1} x^{n-1}y + \cdots + \binom{n}{n-1} xy^{n-1} + y^n \right ) \\ 
        & \leq & \displaystyle 2(n+1) \max \left ( \varphi \left ( \binom{n}{i} x^{n-i} y^i \right ) \right ) \\ 
        & \leq & \displaystyle 2(n+1) \max \left ( 2 \binom{n}{i} \varphi(x)^{n-i} \varphi(y)^i \right ) \\ % not sure if this is correct. "Seems too good"
        & \leq & \displaystyle 4(n+1) \sum_{i=1}^n \left ( \binom{n}{i} \varphi(x)^{n-i}\varphi(y)^i \right ) \\
        & = & 4(n+1) (\varphi(x)+\varphi(y))^n
    \end{array}
\]
We take the $n^{\text{th}}$ root and let $n \rightarrow \infty$, then we get the triangle inequality.
\end{proof}
\begin{thm} Let $\varphi$ be a valuation of $F$ and $T_{\varphi}$ be a topology with the fundamental open sets
    \[
    	B(x,\varepsilon) = \{ y \in F ~|~ \varphi(x-y) < \varepsilon \}
    \]
    then $T_\varphi$ is a Hausdorff space.
\end{thm}
\begin{proof}
     Let $\varphi$ be a valuation of $F$, then for any $\alpha \in \mathbb{R}^{> 0}$, we have that the fundamental open sets of $\varphi$ and $\varphi^\alpha$ are the same. In fact, 
\[
     B_\varphi(x, \varepsilon) = \{ y \in F ~|~ \varphi(x-y) < \varepsilon \} = \{ y \in F ~|~ \varphi^\alpha(x-y) < \varepsilon^\alpha \} = B_{\varphi^\alpha}(x, \varepsilon^\alpha)
\]
This shows that $T_\varphi = T_{\varphi^\alpha}$.
\\ \indent We can assume that $\varphi$ satisfies the triangle inequality because if $||\varphi|| > 2$, then we can always find $\alpha \in \mathbb{R}^{\geq 0}$ such that $||\varphi^\alpha|| = ||\varphi||^\alpha \leq 2$. We first show that the fundamental open sets is actually a basis.
\begin{enumerate}
	\item $x \in B(x,1)$,
    \item Let $y \in B(x, \varepsilon)$, with $\varphi(x-y) = \lambda$. Then we have $0 < \gamma < \varepsilon - \lambda$, so that for any $z \in (y, \gamma)$, we have that
    \[
    	\varphi(x-z) = \varphi(x-y+ y-z) \leq \varphi(x-y) + \varphi(y-z) < \lambda + \varepsilon - \lambda = \varepsilon
    \]
For any two fundamental open sets $B(y_1, \varepsilon_1), B(y_2, \varepsilon_2)$, and a point $x \in B(y_1, \varepsilon_1) \cap B(y_2, \varepsilon_2)$, we just shown that there exists $B(x, \gamma_i) \subset B(y_i, \varepsilon_2)$. Furthermore, $x \in B(x, \gamma_1) \cap B(x, \gamma_2) = B(x, \gamma) \subset B(y_1, \varepsilon_1) \cap B(y_2, \varepsilon_2)$, where $\gamma = \min(\gamma_1, \gamma_2)$.
\end{enumerate}
This shows that the fundamental open sets are indeed a basis and defines a topology on $F$. Now if $x$ and $y$ are two distinct points with $\varphi(x-y) = \varepsilon$, we have the two distinct open sets $B(x, \varepsilon/2)$ and $B(y, \varepsilon/2)$, for any $z$ in the intersection, we have
\[
	\varepsilon = \varphi(x-y) = \varphi(x- z + z-y) \leq \varphi(x-z) + \varphi(z-y) < \varepsilon/2 + \varepsilon/2 = \varepsilon
\]
which gives a contradiction.
\end{proof}
A topology induced by a valuation $\varphi$ will be always be denoted by $T_\varphi$. Viewing the field $F$ as a topological space via $T_\varphi$, we have the following characterization of convergence.
\[
	x_n \rightarrow 0 \text{ in } T_\varphi \text{ if and only if } \varphi(x_n) \rightarrow 0 \text{ in } \mathbb{R}
\]
For one direction, we have that fix $\varepsilon > 0$, then for because $B(0, \varepsilon)$ is an open set containing $0$, there exists $N$ such that $n \geq N \Rightarrow x_n \in B(0, \varepsilon)$. This is translated into
\[
	n \geq N \Rightarrow |\varphi(x_n) - 0| < \varepsilon
\]
hence $\varphi(x_n) \rightarrow 0$ in $\mathbb{R}$. 
\\ \indent Conversely, let $U$ be an open set containing $0$, then because $0$ is an interior point of $U$. Hence there exists $\varepsilon > 0$ such that $0 \in B(0, \varepsilon) \subset U$. For this $\varepsilon$, there exists $N$ such that 
\[
	n \geq N \Rightarrow |\varphi(x_n)| < \varepsilon \Rightarrow x_n \in B(0,\varepsilon) \subset U
\]
Hence the usual metric convergence is equivalent to the topological convergence. Furthermore, we know that in $\mathbb{R}$, we have
\[
	\varphi(x^n) \rightarrow 0 \text{ in } \mathbb{R} \Leftrightarrow \varphi(x) < 1 \Leftrightarrow x \in B(0, 1)
\]
\begin{thm} Let $\varphi_1$ and $\varphi_2$ be nontrivial valuations of $F$ and $x \in F$ be arbitrary, then the following are equivalent
\begin{enumerate}
	\item $\varphi_2 = \varphi_1^\alpha$ for some $\alpha \in \mathbb{R}^{\geq 0}$,
    \item $T_{\varphi_1} = T_{\varphi_2}$,
    \item $T_{\varphi_1} \supset T_{\varphi_2}$, i.e. $T_{\varphi_1}$ is a stronger/finer topology than $T_{\varphi_2}$,
    \item $\varphi_1(x) < 1 \Rightarrow \varphi_2(x) <1$,
    \item $\varphi_1(x) \leq 1 \Rightarrow \varphi_2(x) \leq 1$,
    \item $ \left \{  
    \begin{array}{lcl}
    	\varphi_1(x) < 1 & \Leftrightarrow & \varphi_2(x) <1 \\
        \varphi_1(x) \leq 1 & \Leftrightarrow & \varphi_2(x) \leq 1 \\
        \varphi_1(x) > 1 & \Leftrightarrow & \varphi_2(x) >1
    \end{array} \right .$
\end{enumerate}
\end{thm}
\begin{proof} ~
	\begin{itemize}
    	\item[] (i) $\Rightarrow$ (ii): 
        \begin{multline*}
        	B_2(x, r) = \{ y \in F ~|~ \varphi_2(x-y) < r \} = \{ y \in F ~|~ \varphi_1^\alpha(x-y) < r \} \\ = \{ y \in F ~|~ \varphi_1(x-y) < r^{1/\alpha} \} = B_1(x,r^{1/\alpha})
        \end{multline*}
        The above shows that the fundamental open sets of both topology coincide, so two topologies coincide. 
        \item[] (ii) $\Rightarrow$ (iii): Obvious
        \item[] (iii) $\Rightarrow$ (iv): 
        \[
        	\varphi_1(x) < 1 \Rightarrow x \rightarrow 0 \text{ in } T_{\varphi_1}
        \]
        which implies that any open set $U$ in $T_{\varphi_1}$, there exists $n \geq N$ such that $x_n \in U$. Any open set in $T_{\varphi_2}$ is an open set in $T_{\varphi_1}$, so the same criterion holds. This implies that
        \[
        	x \rightarrow 0 \text{ in } T_{\varphi_2} \Rightarrow \varphi_2(x) < 1
        \]
        \item[] (iv) $\Rightarrow$ (v): Suppose $\varphi_1(x) =1$ and $\varphi_2 (x) >1$ for contradiction. Because $\varphi_1$ is nontrivial, there exists $y \in F$ such that $\varphi_1(y) <1$, so $\varphi_2(y) < 1$. Because $\varphi_2(x) > 1$ there exists $n \in \mathbb{N}$ such that $\varphi_2(1/y) < \varphi_2(x)^n \Rightarrow 1 < \varphi_2(x^ny)$.
        \[
        	\varphi_1(x^ny) = \varphi_1(y) < 1 \Rightarrow \varphi_1(x^ny) <1
        \]
        gives a contradiction.
        \item[] (v) $\Rightarrow$ (vi): Now suppose $\varphi_1(x) < 1$ and $\varphi_2(x) = 1$ and let $y \in F$ such that $\varphi_2(y) >1$. There exists $n \in \mathbb{N}$ such that
        \[
        \varphi_1(x)^n \varphi_1(y) < 1 \Rightarrow \varphi_2(x^ny) <1
        \]
        but $\varphi_2(x^ny) = \varphi_2(y) > 1$, a contradiction. This shows that $\varphi_1(x) < 1 \Rightarrow \varphi_2(x) < 1$. The contrapositive of the assumption gives $\varphi_2(y) > 1 \Rightarrow \varphi_1(y) >1$. Then we have
        \[
        	\varphi_1(x) > 1 \Rightarrow \varphi_1(1/x) < 1 \Rightarrow \varphi_2(1/x) < 1 \Rightarrow \varphi_2(x) >1
        \]
        \[
        	\varphi_2(x) < 1 \Rightarrow \varphi_2(1/x) > 1 \Rightarrow \varphi_1(1/x) >1 \Rightarrow \varphi_1(x) < 1
        \]
        Thus it remains to show that 
        \[
        	\varphi_1(x) = 1 \Leftrightarrow \varphi_2(x) = 1
        \]
        But by the result above $\varphi_1(x) = 1 \Rightarrow \varphi_2(x) \geq 1$ and $\varphi_2(x) \leq 1$, hence $\varphi_2(x) =1$. The analogue proof gives us the desired result.
        \item[] (vi) $\Rightarrow$ (i): The last proof uses the fact that $\mathbb{Q}$ is dense in $\mathbb{R}$. Now let $x \in F$ such that $\varphi_1(x) >1$, and thus $\varphi_2(x) > 1$, then we define 
        \[
        	\alpha = \frac{\log{\varphi_2(x)}}{\log{\varphi_1(x)}}
        \]
        and claim that
        \[
        	\varphi_2 = \varphi_1^{\alpha}
        \]
        Suppose 
        \[
        	\gamma_i = \frac{\log{\varphi_i(y)}}{\log{\varphi_i(x)}}
        \]
        for any $y \in F$. If we could show that $\gamma_1 = \gamma_2$, then we show that $\log{\varphi_1(y)}/\log{\varphi_2(y)} = \log{\varphi_2(x)}/\log{\varphi_1(x)}$, hence for any $y \in F$,
        \[
        	\varphi_1(y)^\alpha = \varphi_1(y)^{\log{\varphi_2(x)}/\log{\varphi_1(x)}} = \varphi_1(y)^{\log{\varphi_2(y)}/\log{\varphi_1(y)}} = \varphi_2(y)
        \]
        so it suffice to show that $\gamma_1 = \gamma_2$. Now let $m/n$ be a rational number such that $m/n \geq \gamma_1$, then 
        \[
        	\begin{array}{lcl}
            	m/n \geq \gamma_1 & \Leftrightarrow & m \log{\varphi_1(x)} \geq n \log{\varphi_1(y)} \\ 
                & \Leftrightarrow & \varphi_1(x^m) \geq \varphi_2(y^n) \\
                & \Leftrightarrow & \varphi_1(x^m/y^n) \geq 1 \\
                & \Leftrightarrow & \varphi_2(x^m/y^n) \geq 2 \\
                & \Leftrightarrow & \varphi_2(x^m) \geq \varphi_2(y^n) \\
                & \Leftrightarrow & m \log{\varphi_2(x)} \geq n \log{\varphi_2(y)} \\
                & \Leftrightarrow & m/n \geq \gamma_2
            \end{array}
        \]
        Performing the same procedure, we get $m/n \leq \gamma_1 \Rightarrow m/n \leq \gamma_2$. By the denseness of $\mathbb{Q}$, we can always fined sequences of $\mathbb{Q}$ $\{ q_i \}$ from above and $\{ t_i \}$ from below that converge to $\gamma_1$. We clearly have that
        \[
        	t_i \leq \gamma_2 \leq q_i
        \]
        with $\lim_{i \rightarrow \infty} t_i = \lim_{i \rightarrow \infty} q_i = \gamma_1$ gives that $\gamma_2 = \gamma_1$ by comparison theorem.
    \end{itemize}
\end{proof}
Two nontrivial valuations are \red{equivalent} if the two induces the same topology. This is clearly a equivalence relation, and by the theorem, we have that its equivalence class is 
\[
	\mfk{p}=[\varphi] = \{ \varphi' ~|~ T_\varphi = T_{\varphi'}\}  = \{ \varphi' ~|~ \varphi' = \varphi^\alpha \text{ for some } \alpha \in \mathbb{R}^{\geq 0} \}
\]
The equivalence class of valuations are called a \red{place}, but, in the case of number theory, places will be called the \red{prime divisor}. The deep connection between the primes and the places will be established in the next section.
\\ \indent Let $\varphi$ and $\varphi'$ be two equivalent valuations, then $\varphi' = \varphi^\alpha$. Then we have that 
\[
	||\varphi'||>1 \Leftrightarrow ||\varphi||^\alpha > 1 \Leftrightarrow ||\varphi|| > 1^{1/\alpha} =1
\]
and similarly, 
\[
	||\varphi'|| = 1 \Leftrightarrow ||\varphi|| = 1
\]
\\ \indent Let $\mfk{p}$ be a prime divisor of $F$, then $\mfk{p}$ is called \red{archimedean} if $||\varphi|| >1$ for all $\varphi \in \mfk{p}$. $\mfk{p}$ is called \red{nonarchimedean} if $||\varphi|| = 1$ for all $\varphi \in \mfk{p}$. Our next goal in this paper will be classifying all prime divisors in a number field $k$. The above remark shows that prime divisors are either archimedean or nonarchimedean.
\\ \indent Let $K, F$ be fields with a field embedding $\mu: F \rightarrow K$. Then for any valuation $\varphi$ of $K$, 
\[
	\varphi^\mu = \varphi \circ \mu
\]
is easily verified to be a valuation of $F$. Furthermore, if $\alpha \in \mathbb{R}^{>0}$, then 
\[
	(\varphi^\alpha)^\mu(x) = \varphi(\mu(x))^\alpha = (\varphi^\mu)^\alpha
\]
hence sends equivalent valuations to equivalent valuations. Hence $\mu$ induces a map $\mu^*$ sending prime divisors of $K$ to prime divisors of $F$ by composing with the valuations of $K$.
\\ \indent Let $\mu: F \rightarrow K$ be a field embedding with $\mfk{q}$ a prime divisor of $K$. Let $\mfk{p}:=\mu^*(\mfk{q})$, then $\mu$ becomes a topological mapping between the topological spaces $(F, T_\mfk{p})$ and $(K, T_\mfk{q})$. We may view $F$ as a subspace of $K$ viewing the homeomorphism $F \cong \mu(F)$, hence the field embedding can be viewed as a topological embedding, thus continuous. In fact, we have $\varphi = \psi \circ \mu$ for some $\psi \in \mfk{q}$ and $\varphi \in \mfk{p}$. $\psi (x-y) < \varepsilon \Rightarrow \psi(\mu(x)- \mu(y)) < \varepsilon$. Furthermore, this shows that 
\begin{align*}
	\{ x_n \} \text{ Cauchy/null sequence in } F & \Rightarrow \{ \mu(x_n) \} \text{ Cauchy/null sequence in } K 
\end{align*}
%%%%%%%%%%%%%%%%%%%%% New Section %%%%%%%%%%%%%%%%%%%%%%%
\section{Nonarchimedean valuation of $k$}
A valuation of a number field $k$ will be sometimes called a \red{multiplicative valuation} or an \red{absolute value} because the concept of valuation is a generalization of the usual absolute value of $\mathbb{Q}$. From now on $F$ will denote arbitrary field, and $k$ will denote number field.
\begin{prp} \label{archchar}
	Let $\varphi$ be a valuation of an arbitrary field $F$, then the following are equivalent
    \begin{enumerate}
    	\item $\mfk{p}=[\varphi]$ is nonarchimedean,
       	\item $\varphi(x+y) \leq \max(\varphi(x), \varphi(y))$ for all $x,y \in F$,
        \item $\{ \varphi(n \cdot 1) ~|~ n \in \mathbb{Z} \}$ is bounded.
    \end{enumerate}
\end{prp}
\begin{proof}
The equivalence of $i)$ and $ii)$ are clear from previous remarks. So we assume $ii)$. Then by induction we have that
\[
	\varphi(x_1+ \cdots + x_n) \leq \max(\varphi(x_1), \ldots, \varphi(x_n))
\]
Then $\varphi(n \cdot 1) \leq \varphi(1)$ which is bounded above $\varphi(1)$.
\\ \indent Now assume $iii)$, then for all $n$, we have $\varphi(n) \leq M$ for some $M$. 
\[
	\begin{array}{lcl}
		[\varphi(x+y)]^n & \leq & \varphi((x+y)^n) \\
    	& = & \varphi \left ( \displaystyle \sum_{i=1}^n \binom{n}{i} x^iy^{n-i}  \right ) \\
        & \leq & 2(n+1) \max \left [ \varphi \left ( \displaystyle \binom{n}{i} x^iy^{n-i} \right ) \right ] \\ 
        & \leq & 2(n+1)M \max (\varphi(x)^i\varphi(y)^{n-i}) \\ 
        & \leq & 2(n+1)M \varphi(x)^n
    \end{array}
\]
where in the last inequality, we assumed that $\varphi(x) \geq \varphi(y)$, then by taking the $n^{\text{th}}$ root of unity, we get
\[
	\varphi(x+y) \leq \sqrt[n]{2(n+1)M} \varphi(x) = \sqrt[n]{2(n+1)M} \max(\varphi(x), \varphi(y))
\]
taking $n \rightarrow \infty$, we see that
\[
	\varphi(x+y) \leq \max(\varphi(x), \varphi(y))
\]
\end{proof}
Now our task is to classify all nonarchimedean prime divisors of $\mathbb{Q}$. We first introduce another form of valuation that resembles the properties of the logarithmic function. An \red{exponential valuation} of $F$ is a function $v: F \rightarrow \mathbb{R} \cup \{ \infty \}$ satisfying,
\begin{enumerate}
	\item $v(x) = \infty \Leftrightarrow x = 0$,
    \item $v(xy) = v(x)v(y)$,
    \item $v(x+y) \geq \min(v(x), v(y))$.
\end{enumerate}
First note that $v$ induces a homomorphism $v:F^* \rightarrow \mathbb{R}$ with $v(F^*)$ a subgroup of additive group $\mathbb{R}$. Let $G \subset \mathbb{R}$ be a subgroup and suppose that $G$ is not discrete, then there exists $x \in G$ that is not an isolated point. This means that for all $\varepsilon \in \mathbb{R}^{>0}$, there exists $y_\varepsilon \in G$ such that $|x - y_\varepsilon| < \varepsilon$. This is again tantamount to saying that for all $\varepsilon \in \mathbb{R}^{>0}$, there exists $0 < y_\varepsilon < \varepsilon$ such that $x \pm y_\varepsilon \in G$. Let $z \in \mathbb{R} \backslash G$ with a fixed $\varepsilon \in \mathbb{R}^{>0}$. We would like to show that $B(z, \varepsilon) \cap G \neq \varnothing$. By the archimedean property we have $n \in \mathbb{N}$, 
\[
	nx \leq y < nx + |x|
\]
then if $z - nx < \varepsilon$, then $nx \in B(z, \varepsilon) \cap G$. So we suppose that $z-nx = \delta > \varepsilon$. There exists $0 < y_\varepsilon < \varepsilon$ and $m$ such that $my_\varepsilon < \delta < (m+1)y_\varepsilon$, so that
\[
	0 \leq \delta - m y_\varepsilon < y_\varepsilon < \varepsilon
\]
so
\[
	|nx + my_\varepsilon - z | = |my_\varepsilon - \delta| < \varepsilon
\]
So we may conclude that $nx + my_\varepsilon \in B(0, \varepsilon)$, so $G$ is dense in $\mathbb{R}$. We thus showed that any subgroup of $\mathbb{R}$ is either discrete of dense in $\mathbb{R}$. We call an exponential valuation $v$ \red{discrete} if $v(F^*)$ is a discrete subgroup of $\mathbb{R}$. 
\\ \indent Furthermore, if $G$ is a discrete subgroup of $\mathbb{R}$, and suppose $G \neq 0$, then there exists $s \in \mathbb{R}$ where $s$ is the smallest positive number in $G$. Unless, $0$ is not an isolated point in $G$. Now suppose there exists $t \in G$, such that $t \neq ns$ for all $n \in \mathbb{Z}$, then for some $n$, we have
\[
	ns < t < (n+1)s
\]
with $0 < t - ns < s$ contradicting the minimality of $s$. Hence all discrete subgroup of $\mathbb{R}$ is of the form $s\mathbb{Z}$. The discrete exponential valuation $v$ is \red{normalized} if $v(F^*) = \mathbb{Z}$ meaning that $s=1$.
\\ \indent The exponential valuation of $\mathbb{Q}$ arises from observing the fundamental theorem of arithmetic and extending to $\mathbb{Q}$, but because our main goal is to study $\mathcal{O}_k$, we give a generalized construction using the unique factorization of ideals. For a fixed $\mfk{p}$ a prime ideal of $\mathcal{O}_k$, for all fractional $\mfk{A}$ of $\mathcal{O}_k$, we have the decomposition
\[
	\mfk{A} = \prod_{\substack{ q \\ \text{ prime ideal }}} \mfk{q}^{e_\mfk{q}(\mfk{A})}
\]
We get a map $\widetilde{v_\mfk{p}} : J_k \rightarrow \mathbb{Z}$ defined by $\mfk{A} \mapsto e_{\mfk{q}}(\mfk{A})$ where $J_k$ is the group of fractional ideals of $\mathcal{O}_k$. This map induces a map $v_\mfk{p}: k \rightarrow \mathbb{Z} \cup \{ \infty \}$ by the composition
\[
	k^* \rightarrow P_k \hookrightarrow J_k \xlongrightarrow{~\widetilde{v_\mfk{p}}~} \mathbb{Z}
\]
with $0 \mapsto \infty$. Suppose $x,y \in F$ with 
\[
	(x) = \prod \mfk{p}^{e_\mfk{p}} \hspace*{4em} (y) = \prod \mfk{p}^{f_\mfk{p}}
\]
where $(x)$ denotes the fractional ideal generated by $x$, then
\begin{enumerate}
	\item $v_\mfk{p}(x) = \infty \Leftrightarrow x=0$ by construction,
    \item We observe that
    \[
    	(xy) = (x)(y) = \prod_{\substack{\mfk{p} \text{ prime ideal }}} \mfk{p}^{e_\mfk{p} + f_\mfk{p}}
    \]
    clearly shows that $v_\mfk{p}(xy) = v_\mfk{p}(x) + v_\mfk{p}(y)$.
    \item We may assume that $e_\mfk{p} \leq f_\mfk{p}$, with $n = e_\mfk{p}$, then we have that
    \[	
    	x,y \in \mfk{p}^n \Rightarrow x+y \in \mfk{p}^n
    \]
    We get the desired result that $v_\mfk{p}(x+y) \geq v_\mfk{p}(x) = \min(v_\mfk{p}(x), v_\mfk{p}(y))$.
\end{enumerate}
Let $v$ be an exponential valuation of $F$, then we may choose $q >1$ and define $\varphi_v: F \rightarrow \mathbb{R}^{\geq 0}$ by
\[
	x \mapsto q^{-v(x)}
\]
which is, in fact, a nonarchimedean valuation of $F$. Conversely, for any nonarchimedean valuation $\varphi$ of $F$, we can define an exponential valuation $v_\varphi: F \rightarrow \mathbb{R} \cup \{ \infty \}$
\[
	x \mapsto - \log{\varphi(x)}
\]
By the remark above, we see that there exists a canonical bijection $v \leftrightarrow \varphi$, and for a fixed $q>1$, we see that another bijection $ v^\alpha \leftrightarrow \varphi^\alpha$. So we will interchangeably use $v \in \mfk{p}$ and $\varphi \in \mfk{p}$ for nonarchimedean prime divisor $\mfk{p}$.
\begin{lem} Let $v$ be a nonarchimedean exponential valuation of a number field $k$, then 
\[
	v = sv_\mfk{p}
\]
for some $s \in \mathbb{R}^{>0}$ and a unique nonzero prime ideal $\mfk{p}$.
\end{lem}
\begin{proof}
We first prove existence of such $\mfk{p}$. Let
\[
	\mfk{p} := \{ \alpha \in \mathcal{O}_k ~|~ v(\alpha) > 0 \}
\]   
We have that $v(\alpha) = \infty >0$, and $v(1) = v(1) + v(1) \Rightarrow v(1) = 0$. Also, $v(1) = v(-1 \cdot -1) = 2v(-1) \Rightarrow v(-1) = 0$. For any $n \in \mathbb{Z}$, $v(n) \geq v(1) = 0$.
\\ \indent Now let $\alpha \in \mathcal{O}_k$ nonzero, then by definition, there exists an equation
\[
	\alpha^n + a_{n-1}\alpha^{n-1} + \cdots + a_0 = 0
\]
with $n$ smallest possible, $a_i \in \mathbb{Z}$. Moving the lower terms to the other side, we get
\[
	\alpha^n = -a_{n-1} \alpha^{n-1} - \cdots - a_0
\]
Let $0 \leq j < n$ with $a_j \neq 0$,
\[
	v(-a_j\alpha^{j}) \geq j v(\alpha)
\]
If $v(\alpha) >0$,
	\[
    	v(\alpha^n) \geq v(-a_{n-1} \alpha^{n-1} - \cdots - a_0) \geq v(-a_j \alpha^{j}) \geq jv(n) \geq (n-1)v(n)  
    \]
and we have $n v(\alpha) = v(\alpha^n)$, so we get
\[
	n v(\alpha) \geq (n-1) v(\alpha) \Rightarrow v(\alpha) \geq 0
\]
which yields a contradiction. This proves that for all $\alpha \in \mathcal{O}_k$, $v(\alpha) \geq 0$.
\\ \indent Next step is to show that $\mfk{p}$ is a prime ideal. But this is easy to see as $x,y \in \mfk{p}$, $v(x+y) \geq min(v(x), v(y)) >0$, hence $x+y \in \mfk{p}$, and for all $x \in \mathcal{O}_k$ and $y \in \mfk{p}$, we have $v(x) \geq 0$, $v(y) >0$. $v(xy) = v(x) + v(y) >0$. Finally, $v$ is not identically $0$ for $k^*$ because $v$ is not trivial. $v$ is not identically $0$ for all $\mathcal{O}_k \backslash \{ 0 \}$. In fact, if it was, then for any $x/y \in k^*$, $v(x/y) = v(x) - v(y) = 0 - 0 =0$. We have shown that $\mfk{p}$ is nonzero.
\\ \indent Suppose $v_\mfk{p}(\alpha) = 0$, then we will show that $v(\alpha) = 0$. To prove this we need another lemma.
\begin{lem} Let $\mfk{p} \subset \mathcal{O}_k$ be a prime ideal with $\alpha \in k^*$, then $\alpha = a/b$ with $v_\mfk{p}(b) = 0$ for some $a,b \in \mathcal{O}_k$.
\end{lem}
\begin{proof}
	Let $(\alpha) = \mfk{A}\mfk{B}^{-1}$ where $\mfk{A}, \mfk{B}$ is an ideal of $\mathcal{O}_k$ with $(\mfk{A}, \mfk{B}) = \mathcal{O}_k$, i.e. relatively prime. Then we have $v_\mfk{p}(\mfk{B}) = 0$ as $v_\mfk{p}(\alpha) \geq 0$. Because $\mfk{B}(\alpha) = \mfk{A}$, we have that $[\mfk{A}] = [\mfk{B}]$, where $[~~]$ denotes the ideal class.
    \\ \indent Let $[\mfk{A}]$ be any arbitrary ideal class, then $\mfk{A}$ has an unique decomposition,
    \[
    	\mfk{A} = \prod_{i=1}^n \mfk{p}^{e_{\mfk{p}_i}}
    \]
    with $e_{\mfk{p}_i} \in \mathbb{Z}$. Now fix a $\mfk{p}$ prime ideal with and let $\pi \in \mfk{p}^m \backslash \mfk{p}^{m+1}$, then we have
    \[
    	(\pi) = \mfk{p}^m \prod \mfk{q}^{f_{\mfk{q}_i}}
    \]
    where $m = e_{\mfk{p}_i}$. By chinese remainder theorem, we have an element $x$ relatively prime to $\mfk{p}$ and divisible by $\mfk{p}_i, \mfk{q}_i$ that is not $\mfk{p}$, then taking sufficiently large $l$, we get
    \[
    	(\pi^{-1}x^l)\mfk{A}
    \]
    represents $[\mfk{A}]$, is an integral ideal, and is relatively prime to $\mfk{p}$. 
    \\ \indent Let $\mfk{C} \in [\mfk{A}]^{-1} = [\mfk{B}]^{-1}$ with $\mfk{C}$ an integral ideal relatively prime to $\mfk{p}$, then we get
    \[
    	(\alpha) = \mfk{A}\mfk{C} (\mfk{B}\mfk{C})^{-1}
    \]
$\mfk{A}\mfk{C} = (a), \mfk{B}\mfk{C} = (b)$ giving us 
\[
	\alpha = \frac{a}{b}
\]
with $v_\mfk{p}(b) = 0$ as both $\mfk{B}, \mfk{C}$ are relatively prime to $\mfk{p}$.
\end{proof}
By the lemma, we easily get
\[
	v_\mfk{p}(\alpha) = v_\mfk{p}(a/b) = v_\mfk{p}(a) - v_\mfk{p}(b) = v_\mfk{p}(a) = 0
\]
$a,b \in \mathcal{O}_k$ and $a,b \notin \mfk{p}$, we have $v(a) = v(b) = 0$ by definition. For any $\alpha \in k^*$, $n = v_\mfk{p}(\alpha) \in \mathbb{Z}$, and for $\pi \in \mfk{p} \backslash \mfk{p}^2$, we get $v_\mfk{p}(\pi) =1$. $v_\mfk{p}(\alpha \pi^{-n}) = 0$, so $v(\alpha \pi^{-n}) = 0$.
\[
	v(\alpha) = v(\pi^n) = v(\pi) n = v(\pi) v_\mfk{p}(\alpha)
\]
Let $s= v(\pi)$ which is positive as $\pi \in \mfk{p}$, we get the desired result.
\\ \indent Now for uniqueness, let $v = s_1v_\mfk{p_1} = s_2 v_\mfk{p_2}$, then 
\begin{multline*}
	\mfk{p_1} = \{\alpha \in \mathcal{O}_k ~|~ v_\mfk{p_1}(\alpha) > 0 \} = \{\alpha \in \mathcal{O}_k ~|~ s_1 v_\mfk{p_1}(\alpha) > 0 \} \\ = \{\alpha \in \mathcal{O}_k ~|~ s_2 v_\mfk{p_2}(\alpha) > 0 \} = \{\alpha \in \mathcal{O}_k ~|~ v_\mfk{p_1}(\alpha) > 0 \} = \mfk{p_2}
\end{multline*}
hence the prime ideal $\mfk{p}$ is unique.
\end{proof}
We thus have shown that the only nonarchimedean prime divisor $\mfk{p}$ of a number field $k$ is $[v_\mfk{p}]$. Our next goal would be to classify all archimedean valuations of $k$,but we need completion which is a generalized technique used to construct the real numbers from $\mathbb{Q}$.
\section{Completion}
Let $\mfk{p}$ be a prime divisor of a field $F$, then a $\mfk{p}$-Cauchy sequence is a sequence $\{x_n\}$ in $F$ such that
\[
	x_n - x_m \rightarrow 0 \text{ as } n, m \rightarrow \infty \text{ in } T_\mfk{p}
\]
$F$ is said to be \red{$\mfk{p}$-complete} if all $\mfk{p}$-Cauchy sequences converge. A field extension $\widehat{F}$ of $F$ with a prime $\widehat{\mfk{p}}$ with a field embedding $\mu:F \rightarrow \widehat{F}$ is said to be a \red{$\mfk{p}$-completion} of $F$ if
\begin{enumerate}
	\item $F \subset \widehat{F}$,
    \item $\mu(F) \subset \widehat{F}$ is dense,
    \item $(\widehat{F}, \widehat{\mfk{p}})$ is $\widehat{\mfk{p}}$-complete.
\end{enumerate}
Notice that once we have found $\widehat{F}$ satisfying ii) and iii), we may use set-theoretical trick to make $\widehat{F}$ into an extension of $F$. We find $S$ which is a set bijective to $\widehat{F} - F$ with respect to a bijection $\psi$. Let $\widetilde{F} = S \cup F$, then the bijection $\psi$ can be extended to $\Psi: \widehat{F} \rightarrow \widetilde{F}$. Then $\widetilde{F}$ is an extension of $F$ by for all $x,y \in \widetilde{F}$,
\[
	xy :=\Psi(\Psi^{-1}(x)\Psi^{-1}(y)) \text{ and } x+y :=\Psi(\Psi^{-1}(x)+\Psi^{-1}(y))
\]
\begin{thm} Let $F$ be a field with a prime divisor $\mfk{p}$. A $\mfk{p}$-completion of $F$ exists, and it is essentially unique. In other words, if 
\[
	(\widehat{F_1}, \widehat{\mfk{p}_1}, \mu_1) \text{ and } (\widehat{F_2}, \widehat{\mfk{p}_2}, \mu_2)
\]
are two $\mfk{p}$-completions of $F$, then there exists an unique isomorphism
\[
	\sigma: \widehat{F_1} \rightarrow \widehat{F_2}
\]
which is an extension of $1_F: F \rightarrow F$ such that $\sigma^*(\widehat{\mfk{p}_2}) = \widehat{\mfk{p}_1}$.
\end{thm}
\begin{proof}
	We choose $\varphi \in \mfk{p}$ such that $||\varphi|| \leq 2$. In other words, $\varphi$ satisfies the triangle inequality. We let $\mathcal{R}$ be the set of all $\mfk{p}$-Cauchy sequences in $F$, then addition and multiplication defined componentwise makes $\mathcal{R}$ into a ring as seen in \Cref{cauchyring}. Let $X = \{ x_n \} \in \mathcal{R}$, then there exists a canonical way to extend $\varphi$ by putting
    \[	
  		\widetilde{\varphi}(X) = \lim_{n \rightarrow \infty} \varphi(x_n)
    \]
    then the reverse triangle inequality \Cref{reversetri} shows that the $\varphi(x_n)$ is Cauchy,
    \[
    	|\varphi(x_n) - \varphi(x_m)| \leq \varphi(x_n - x_m)
    \]
    and the fact that $\mathbb{R}$ is complete shows that the limit converges. We can easily see that for any $X,Y \in \mathcal{R}$,
    \begin{enumerate}
    	\item $\widetilde{\varphi}(X) \geq 0$, as $\varphi(x_n) \geq 0$ for all $n$.
        \item $\widetilde{\varphi}(XY) = \widetilde{\varphi}(x)\widetilde{\varphi}(Y)$, as
        \[
        	\widetilde{\varphi}(XY) = \lim_{n \rightarrow \infty} \varphi(x_n) \varphi(y_n) = \lim_{n \rightarrow \infty} \varphi(x_n) \lim_{n \rightarrow \infty} \varphi(y_n)  = \widetilde{\varphi}(X) \widetilde{\varphi}(Y)
        \]
        \item $\widetilde{\varphi}(X+Y) \leq \widetilde{\varphi}(X) + \widetilde{\varphi}(Y)$ as 
        \[
        	\widetilde{\varphi}(X+Y) = \lim_{n \rightarrow \infty} \varphi(x_n + y_n) \leq \lim_{n \rightarrow \infty} \varphi(x_n) + \varphi(y_n) = \widetilde{\varphi}(X)\widetilde{\varphi}(Y)
        \]
    \end{enumerate}
\noindent hence $\widetilde{\varphi}$ defines a valuation of $\mathcal{R}$. Next we define the subset of $\mathcal{N}$ of $\mathcal{R}$ by 
\[
	\mathcal{N}:= \{ X \in \mathcal{R} ~|~ X=\{x_n\}, \lim_{n \rightarrow \infty} \varphi(x_n) = 0 \}
\]
called the \red{null sequences} of $F$. We define $\widehat{F}:= \mathcal{R}/\mathcal{N}$ and see that it is a field is proven similarly to \Cref{inversenull}. Also, let $N \in \mathcal{N}$ be a null sequence, then we have
\[
	\widetilde{\varphi}(A) \leq \widetilde{\varphi}(A+N) + \widetilde{\varphi}(N) = \widetilde{\varphi}(A+N) \leq \widetilde{\varphi}(A)
\]
shows that $\widetilde{\varphi}$ is well-defined for the cosets of the form $X + \mathcal{N}$. We have that $\widetilde{\varphi}$ is a valuation for $\widehat{F}$. We let $\widehat{\mfk{p}} = [\widetilde{\varphi}]$. Also, we have a monomorphism $\mu: F \rightarrow \widehat{F}$ by $x \mapsto [x]$ where $[x_n]$ stands for the equivalence class of $\{ x_n \}$ in $\widehat{F}$. We have that $\mu^*(\widehat{\mfk{p}}) = \mfk{p}$ because
\[
	\widetilde{\varphi}(\mu(x)) = \lim_{n \rightarrow \infty} \varphi(x) = \varphi(x)
\]
Now consider $[X] \in \widehat{F}$ with $X = \{ x_n \} \in \mathcal{R}$. For each $n$, $[X]- \mu(x_n)$ is represented by the sequence
\[
	\{ x_1 - x_n, x_2 - x_n, \ldots \}
\]
Then by using the property above,
\[
	\lim_{n \rightarrow \infty} \widetilde{\varphi}([X] - \mu(x_n)) = \lim_{n \rightarrow \infty} \left ( \lim_{m \rightarrow \infty} \varphi(x_m - x_n) \right ) = 0
\]
as $\{x_n\}$ is a Cauchy sequence. This shows that $\mu(F)$ is dense in $\widehat{F}$. We have found an element in $F$ that is arbitrary close to a point in $\widehat{F}$.
\\ \indent Finally consider a Cauchy sequence $\{ [X]_n \}$ of $\widehat{F}$, then by density, there exists $y_n \in F$ such that 
\[
	\widetilde{\varphi}([X]_n - \mu(y_n)) < 1/n
\]
for each $n$, then we have $\{ [X]_n - \mu(y_n) \}$ is a null sequence which shows that $\{ \mu(y_n) \}$ is a Cauchy sequence in $\widehat{F}$. Then because $\widetilde{\varphi}^\mu = \varphi$, we have that $\{ y_n \}$ is a Cauchy sequence in $F$.  Let $[Y] \in \widetilde{F}$ with $Y = \{ y_n \}$, then we have
\[
	\widetilde{\varphi}([Y] - \mu(y_n)) = 0 \Rightarrow \widetilde{\varphi}([Y]- [X]_n) = 0
\]
hence $\widehat{F}$ is $\widehat{\mfk{p}}$ complete. The uniqueness is due to the following more general proposition,
\end{proof}
\begin{prp} \label{incluison-valuation} Consider $K, F$ be two fields with prime divisors $\mfk{q}$ and $\mfk{p}$ respectively. Suppose there exists a field embedding $\sigma: F \rightarrow K$ such that $\sigma^*(\mfk{q}) = \mfk{p}$. Let $(\widehat{F}, \widehat{\mfk{p}}, \mu)$ and $(\widehat{K}, \widehat{\mfk{q}}, \lambda)$ be $\mfk{p}, \mfk{q}$ completion of $F, K$ respectively, then there exists a unique monomorphism $\widehat{\sigma}: \widehat{F} \rightarrow \widehat{K}$ such that
\begin{enumerate}
	\item $(\widehat{\sigma})^*(\widehat{\mfk{q}}) = \widehat{\mfk{p}}$,
    	\begin{center}
    	\begin{tikzpicture}
        	\matrix (m) [matrix of math nodes, nodes={anchor=center}, row sep = 3em, column sep= 3em]{
            	\widehat{F} & \widehat{K} \\
                F & K \\
            };
            \draw[->]
            	(m-1-1) edge node[above]{$\widehat{\sigma}$} (m-1-2)
                (m-2-1) edge node[above]{$\sigma$} (m-2-2)
                		edge node[left]{$\mu$} (m-1-1)
                (m-2-2) edge node[right]{$\lambda$} (m-1-2);
               
        \end{tikzpicture}
    \end{center}
    \item $\widehat{\sigma} \circ \mu = \mu' \circ \sigma$,
\end{enumerate}
\end{prp}
\begin{proof} Let $x \in \widehat{F}$. By denseness of $F$, there exists a Cauchy sequence $\{ x_n \}$ in $F$ such that $\mu(x_n) \rightarrow x$ in $\widehat{F}$. By assumption we have $\mu^*(\widehat{\mfk{p}}) = \mfk{p}, \sigma^*(\mfk{q}) = \mfk{p}, \lambda^*(\widehat{q}) =\mfk{q}$.
\[
	\{ x_n \} \text{ Cauchy in } T_\mfk{p} \Leftrightarrow \{ \sigma(x_n) \} \text{ Cauchy in } T_\mfk{q} \Leftrightarrow \{ \lambda(\sigma(x_n)) \} \text{ Cauchy in } T_{\widehat{q}}
\]
as $\widehat{K}$ is complete, we not define $\widehat{\sigma}$ by
\[
	\widehat{\sigma}(x) = \lim_{n \rightarrow \infty} \lambda(\sigma(x_n))
\]
\begin{enumerate}
	\item Suppose $\{ y_n \}$ is another Cauchy sequence that converges to $x$. We have that $\{ x_n - y_n \}$ is a null sequence in $T_\mfk{p}$, then $\{ \lambda(\sigma(x_n - y_n)) \}$ is a null sequence in $T_{\widehat{\mfk{q}}}$. Hence
    \[
    	\lim_{n \rightarrow \infty} \lambda(\sigma(x_n)) = \lim_{n \rightarrow \infty} \lambda(\sigma(y_n))
    \]
    \item We have that $\sigma, \lambda, \lim_{n \rightarrow \infty}$ has the additive, multiplicative property, hence $\widehat{\sigma}$ is a homomorphism,
    \item 
    \[
    	\widehat{\sigma}\mu(x) = \lim_{n \rightarrow \infty} \lambda(\sigma(x)) = \lambda(\sigma(x))
    \]
    shows that the diagram commutes.
    \item Let $\widehat{\sigma}(x) = \widehat{\sigma}(y)$, then we have $\lim_{n \rightarrow \infty} \lambda(\sigma(x_n)) = \lim_{n \rightarrow \infty} \lambda(\sigma(y_n)) \Rightarrow x_n-y_n$ is a null sequence in $F$, hence $\lim_{n \rightarrow \infty} \mu(x_n-y_n) = 0 \Rightarrow x = \lim_{n \rightarrow \infty} \mu(x_n) = \lim_{n \rightarrow \infty} \mu(y_n) = y$.
\end{enumerate}
Next step is to prove that $(\widehat{\sigma})^*(\widehat{\mfk{q}}) = \widehat{\mfk{p}}$. Let $\varphi \in \widehat{\mfk{q}}$, then we use that notation $\varphi^{\lambda}$ for $\varphi \circ \lambda$, then $\varphi^{\lambda \sigma} \in \sigma^*( \lambda^*(\widehat{\mfk{q}})) = \sigma^*(\mfk{q}) = \mfk{p}$. As $\mu^*(\widehat{\mfk{p}}) = \mfk{p}$, there exists $\theta \in \widehat{p}$ such that $\theta^\mu = \varphi^{\lambda \sigma}$. Now let $x \in F$, then we have that
\[
	x = \lim_{n \rightarrow \infty} \mu(x_n)
\]
\begin{align*}
	\varphi^{\widehat{\sigma}}(x) & = \varphi(\widehat{\sigma}(x)) \\
    & = \varphi \left ( \lim_{n \rightarrow \infty} \lambda(\sigma(x_n)) \right ) \\ 
    & = \lim_{n \rightarrow \infty} \varphi^{\lambda \sigma} (x_n) \\
    & = \lim_{n \rightarrow \infty} \theta^\mu(x_n) \\
    & = \theta \left ( \lim_{n \rightarrow \infty} \mu(x_n) \right ) = \theta(x)
\end{align*}
hence we have proved that $\varphi^{\widehat{\sigma}} \in \widehat{\mfk{p}}$. In particular, $(\widehat{\sigma})^*(\widehat{\mfk{q}}) = \widehat{\mfk{p}}$.
\\ \indent As we have proved existence, we are left to show uniqueness. Let $\widehat{\sigma}_1, \widehat{\sigma}_2$ be two extensions of $\sigma$ with the desired properties, then $\widehat{\sigma}_1(\mu(F)) = \widehat{\sigma}_2(\mu(F))$ with $\mu(F)$ dense in $\widehat{F}$. Because a continuous mapping agrees in a dense set, and $\widehat{K}$ is Hausdorff, we conclude that $\widehat{\sigma}_1 = \widehat{\sigma}_2$.
\end{proof}
Let every exponential valuation of $v$ of $K$, we have a completion $K_v$ which extends canonically to a exponential valuation of $K_v$ denoted again by $v$. There exists a unique extension $\overline{v}$ to $\overline{K_v}$, the algebraic closure of $K_v$. The unique extension exists due to the following theorem.
\begin{thm} Let $F$ be complete with respect to the valuation $\varphi$. If $K/F$ is an algebraic extension, $\varphi$ is extended to a unique extension of $\varphi$ to a valuation of $L$. If $[K:F]<\infty$, then the extension is given by
\[
	\varphi'(x) = \sqrt[n]{\varphi(N^K_F(x))}
\]
and $L$ is complete respect to the extension.
\end{thm}
\begin{proof}
	 We start with $[K:F]< \infty$. Clearly $N^K_F(x) = x^n$ for $x \in F$, hence $\varphi'|_K = \varphi$. Because $\varphi(y) \geq 0$, $\varphi'(x) \geq 0$, and $\sqrt[n]{x} = 0, \varphi(x) = 0, N_F^K(x) = 0 \Leftrightarrow x=0$ shows that $\varphi'(x)=0 \Leftrightarrow x=0$. Also all three functions, $\sqrt[n]{\color{white}x}, \varphi, N_F^K$ are all multiplicative implies that $\varphi'$ multiplicative. So it remains to show that $\varphi'(x) \leq 1 \Rightarrow \varphi'(1+x) \leq 1$.
     \\ \indent Obviously, $\varphi'(x) \leq 1 \Rightarrow \sqrt[n]{\varphi(N_F^K(x))} \leq 1 \Rightarrow \varphi(N_F^K(x)) \leq 1$. Let $f = m_F(x)(t)$ be the minimal polynomial of $x$ over $F$ with $f = t^r + \cdots + a_0$. Then $N_F^K(x) = \pm a_0^d$ where $d=[K:F(x)]$ by \Cref{normlem}. This gives that $\varphi(a_r)^d \leq 1 \Rightarrow a_r \in \mathcal{O} \Rightarrow f \in \mathcal{O}[t]$ by \Cref{henlem}.  We now put $g(t) = f(t-1)$, then $g(1+x) = f(x) = 0$ with $g$ minimal polynomial of $1+x$ over $F$. 
     \[
     	N^K_F(1+x) = (\pm g(0))^d = (\pm f(-1))^d = \pm ((-1)^d + \cdots +a_r) \in \mathcal{O} \Rightarrow \varphi(N_F^K(1+x)) \leq 1
     \]
     $\varphi'(1+x) \leq 1$ then easily follows.
\\ \indent If $K/F$ is infinite, then we observe that 
\[
	K = \bigcup_E E
\]
where $E$ runs over all finite intermediate field of $K/F$. For each $E$, there exists an extension $\varphi_E$ of $E$. Let $x \in K$, then $x \in E$ for some $E$. We define $\varphi'(x) = \varphi_E(x)$. To finish the proof we need uniqueness which is dealt in a more general setting.
\end{proof}
An finite dimensional $F$-vector space $X$ is said to be \red{normed} over $(F, \varphi)$ if there exists a norm function $\|{\color{white}x}\|: X \rightarrow \mathbb{R}$, satisfying,
\begin{enumerate}
	\item $\| \zeta \| \geq 0$ and $\| \zeta \| = 0 \Leftrightarrow \zeta = 0$,
    \item $\|x \zeta \| = \varphi(x) \| \zeta \|$,
    \item $\| \zeta + \eta \| \leq \| \zeta \| + \| \eta \|$.
\end{enumerate}
Let $(E, \varphi')$ be a extension of $(F, \varphi)$, then the additive group $E$ with $\|{\color{white}x} \| = \varphi'$ is normed over $(F, \varphi)$. Also, if $X$ is a finite dimensional vector space over $F$ with $w_1, \ldots, w_n$ as its basis, then for all $\zeta \in X$, $\zeta = a_1 w_1 + \cdots + a_n w_n$. Define
\[
	\|\zeta\| = \max(\varphi(a_i))
\]
which runs over all $i$, then $\|{\color{white}x}\|$ becomes a norm of $X$ and call it the \red{canonical norm}.
\begin{thm} Let $F$ be complete respect with $\mfk{p}$ and let $\varphi \in \mfk{p}$ with $\| \varphi \| \leq 2$. Suppose that $X$ is a finite dimensional vector space over $F$, and $|{\color{white}x}|$ be any norm on $X$ over $(F, \varphi)$, then $X$ is complete topological group in the metric topology determined by $|{\color{white}x}|$ and there exists $C_1, C_2 >0$ such that
\[
	C_1 \| \zeta \| \leq |\zeta| \leq C_2 |\zeta|
\]
for all $\zeta \in X$.
\end{thm}
\begin{proof} Let $\{ \zeta_m \}$ be Cauchy sequence respect to $|{\color{white}x}|$, then for all $\varepsilon \in \mathbb{R}^{>0}$, we have
\[
	\| \zeta_{m_1} - \zeta_{m_2} \| < \varepsilon \text{ whenever } m_1, m_2 \geq N
\]
with $\zeta_m = a_{m,1}w_1 + \cdots + a_{m,n}w_n$. We then have a natural inequality,
\[
	\varphi(z_{m_1, i} - z_{m_2, i}) \leq \|z_{m_1} - z_{m_2} \| < \varepsilon
\]
shows that each component is a Cauchy sequence in $F$, thus there exists $a_j$ such that $a_{m, j} \rightarrow a_j$. Let $\zeta = a_1w_1 + \cdots + a_nw_n$, then we claim that $\zeta_m \rightarrow \zeta$. For any $\varepsilon$, and for all $i$, there exists $N_i$ such that
\[
	m \geq N_i \Rightarrow \varphi(a_{m,i}-a_i) < \varepsilon
\]
then if $N = \max(N_i)$, implies that
\[
	m \geq N \Rightarrow \| \zeta_m - \zeta \| = \max( \varphi(a_{m,i}-a_i)) < \varepsilon
\]
which proves the claim. Furthermore, proves that $X$ is complete with respect to $\|{\color{white}x}\|$. 
\\ \indent Let $\zeta = a_1 w_1 + \cdots a_n w_n$, then
\[
	\begin{array}{lcl}
    	|\zeta| & = & |a_1w_1 + \cdots + a_nw_n| \\
        & \leq & |a_1w_1| + \cdots + |a_nw_n| \\
        & = & \varphi(a_1) |w_1|+ \cdots + \varphi(a_n) |w_n| \\
        & \leq & \|\zeta\| |w_1| + \cdots + \| \zeta \| |w_n| = C_2 \| \zeta \|
    \end{array}
\]
if we let $C_2 = |w_1| + \cdots + |w_n|$. If $n=1$, then $\varphi = |{\color{white}x}|$, hence the result is immediate. Suppose that we have  the conclusion for $n-1$-dimensional. Then we define
\[
	Y_i = Fw_1 + \cdots + Fw_{i-1} + Fw_{i+1} + \cdots + Fw_n
\]
then $X = Y_i + Fw_i$. We have that $Y_i$ is complete by induction hypothesis. $Y_i$ complete implies that $Y_i$ closed in $X$ as all limit point of $Y_i$ has a Cauchy sequence in $Y_i$. Since $X$ is a topological group, $Y_i + w_i$ is closed. Because $0 \notin Y_i + w_i$, so $0 \notin \cup_{i=1}^n (Y_i + w_i)$ which is also closed. There exists a neighborhood of $0$ disjoint from the union. Hence there exist $C_1$ such that
\[
	|\eta_i + w_i| \geq C_1
\]
for all $\eta_i \in Y_i$. If $\zeta = a_1w_1 + \cdots + a_nw_n|$ nonzero with $\varphi(a_r) \|\zeta\|$, then $a_r \neq 0$, and 
\[
	|a_r^{-1}| = |(a_1/a_r)w_1 + \cdots + w_r + \cdots + (a_n/a_r)w_n| \geq C_1
\]
hence we have that $|\zeta| \geq C_1 \varphi(a_r) = C_1 \| \zeta \|$.
\\ \indent For all $y \in B_{|~|}(x, \varepsilon)$, $y \in B_{\|~\|}(x, \varepsilon/C_1) \subset B_{|~|}(x, \varepsilon)$, and similar process shows that the topology induced by both norms are equivalent.
\end{proof}
\begin{prp} Let $F$ be $\mfk{p}$-complete and $K/F$ be finite. If an extension of $\mfk{p}$ exists, then it is unique. Furthermore, if $K/F$ be algebraic extension, then the extension of $\mfk{p}$ is unique if it exists.
\end{prp}
\begin{proof}
First for the finite case, we choose $\varphi \in \mfk{p}$ such that $\| \varphi \| \leq 2$, with $\varphi_1, \varphi_2$ be two extensions of $\varphi$ to $K$. Then by above, $\varphi_1 = \varphi_2^\alpha$ because they induce the same topology. Furthermore, $\varphi_1|_F = (\varphi_2|_F)^\alpha \Rightarrow \varphi = \varphi^\alpha \Rightarrow \alpha =1$. We thus see that they are equal. 
\\ \indent Now for algebraic extension $K/F$, let $\varphi_1, \varphi_2$ be two distinct extensions of $\varphi$ to $K$, then there exists $x \in K$ such that $\varphi_1(x) \neq \varphi_2(x)$, but $\varphi_1$ and $\varphi_2 $ coincide in $F(x)$. Hence leads to contradiction.
\end{proof}
Let $K/F$ be finite with $\varphi \in \mfk{p}$ be valuation of $F$ and $\psi \in \mfk{q}$ be valuation of $K$ extending $\varphi$. Also, we have the canonical extension $\widetilde{\varphi}$ of $\varphi$ to $F_\mfk{p}$. We also have unique extension $\obar{\varphi}$ of $\widetilde{\varphi}$ to $\obar{F_\mfk{p}}$.
\\ \indent We first show that $F_\mfk{p} \subset K_\mfk{q}$. If we denote $[x_n]_\mfk{p} \in F_\mfk{p}$ for the $\mfk{p}$-Cauchy sequence with $x_n \in F$. Similarly, $[x_n]_\mfk{q}$ denote $\mfk{q}$-Cauchy sequence with $x_n \in K$. As $F \subset K$, we consider the map $[x_n]_\mfk{p} \mapsto [x_n]_\mfk{q}$. For $x_n, y_n \in F$,
\[
	[x_n]_\mfk{p} = [y_n]_\mfk{p} \Leftrightarrow  \lim_{n \rightarrow \infty} \varphi(x_n - y_n) = 0 \Leftrightarrow \lim_{n \rightarrow \infty} \psi(x_n - y_n) = 0  \Leftrightarrow [x_n]_\mfk{q} = [y_n]_\mfk{q}
\]
show that the map is well-defined and injective. Hence we proved the inclusion of $F_\mfk{p}$ into $K_\mfk{q}$.
\\ \indent $F_\mfk{p}/F$ and $K/F$ finite implies that $KF_\mfk{p}/F_\mfk{p}$ is finite. As $K$ and  $F_\mfk{p}$ both lies inside $K_\mfk{q}$, we have $KF_{\mfk{p}} \subset K_\mfk{q}$. If $\widetilde{\psi}$ denote the canonical extension of $\psi$, then we give $KF_{\mfk{p}}$ the valuation $\widetilde{\psi}|_{KF_\mfk{p}}$, then since it extends $\psi$, we see that it must be uniquely determined and $KF_\mfk{p}$ is complete. Observing the tower of fields,
\[
	K \subset KF_\mfk{p} \subset K_\mfk{q}
\]
we see that every completing via $\mfk{q}$, we get
\[
	K_\mfk{q} \subset KF_{\mfk{p}} \subset K_\mfk{q} \Rightarrow K_\mfk{q} = KF_\mfk{p}
\]
Hence $K_\mfk{q}/F_\mfk{p}$ is a finite extension. We are now ready to prove the extension theorem which is the following
\begin{thm}[Extension Theorem] Let $K/F$ be a finite extension and $\varphi \in \mfk{p}$ be a valuation of $F$, then every extension $\psi$ of $\varphi$ arises as the composite $\psi = \overline{\varphi} \circ \tau$ for some embedding $\tau: K \rightarrow \obar{F_\mfk{p}}$.
\end{thm}
\begin{proof}
	With the above notation, $K_\mfk{q} \subset \obar{F_\mfk{p}}$ shows that $\obar{\varphi}|_{K_\mfk{q}} = \widetilde{\psi}$ as $\widetilde{\psi}$ is the unique extension of $\widetilde{\varphi}$. Because $K_\mfk{q}/F_\mfk{p}$ finite, we have a $F_\mfk{p}$-embedding,
    \[
    	\tau : K_\mfk{q} \rightarrow \obar{F_\mfk{p}}
    \]
    Clearly, $\varphi \circ \tau$ is a valuation of $K_\mfk{q}$, and $(\obar{\varphi} \circ \tau)|_{F_\mfk{p}} = \obar{\varphi} \circ \tau|_{F_\mfk{p}} = \obar{\varphi}|_{F_\mfk{p}} = \widetilde{\varphi}$. We have a $F$-embedding, $\tau|_{\mfk{p}}:K \rightarrow \obar{F_\mfk{p}}$,
    \[
    	\widetilde{\varphi} = \obar{\varphi} \circ \tau|_{F_\mfk{p}}
    \]
\end{proof}
Before we classify the archimedean valuations of a number field $k$, we state the theorem due to Ostrowski.
\begin{thm}[Ostrowski] Let $K$ be complete respect to an archimedean valuation $|{\color{white} x}|$, then there exists an isomorphism $\sigma$ of $K$ to either $\mathbb{R}$ or $\mathbb{C}$. If $\mfk{p}_\infty$ denotes the prime divisor of $|{\color{white}x}|$ and $\mfk{q}_\infty$ denotes the prime divisor of the usual absolute value of either $\mathbb{R}$ or $\mathbb{C}$, then $\sigma^*(\mfk{q}_\infty) = \mfk{p}_\infty$. 
\end{thm}
\begin{proof} $K$ having archimedean valuation implies that $K$ has characteristic $0$, the prime field of $K$ is isomorphic to $\mathbb{Q}$. By \Cref{incluison-valuation}, we may assume that $\mathbb{R} \subset K$. Next we prove that $K$ is algebraic over $\mathbb{R}$, then $K$ is embedded in $\mathbb{C}$, hence $K= \mathbb{R}$ or $K = \mathbb{C}$.
\\ \indent So consider $\zeta \in K$ and a continuous map $f: \mathbb{C} \rightarrow \mathbb{R}$ defined by 
\[
	f(z) = |\zeta^2 - (z+\overline{z})\zeta+ z \overline{z}|
\]
with $z+\overline{z}, z \overline{z} \in \mathbb{R}$. We first observe that $ \displaystyle \lim_{z \rightarrow \infty} f(z) = \infty$. If we let $\zeta = a+bi$, then 
\[
	f(z) = |(a+bi)^2 - 2Re(z)(a+bi) + Re(z)^2 + Im(z)^2|
\]
then $Im(z) \rightarrow \infty$ easily shows that $f(z) \rightarrow \infty$.
\\ \indent Let $z_0$ be any complex number, then let $M=f(z_0)$, then there exists $N$ such that $|z| > N \Rightarrow |f(z)| > M$. There exists $z_m \in \mathbb{C}$ such that $f(z_m)$ is minimal in $f(\overline{B(0, N)})$ as $\overline{B(0,N)}$ is compact. Hence there exists a minimum value for $f$. We let $f(z_m) = m$. Because $\{m\}$ is closed in $\mathbb{R}$, we have
\[
	S = \{ z \in \mathbb{C} ~|~ f(z) = m \}
\]
is nonempty, bounded, and closed as $f$ is continuous. Also because $S$ is closed, there exists $z_0$ such that $|z_0| \geq |z|$ for all $z \in S$, i.e. the supremum. It suffices to show that $m=0$, then $\zeta$ is algebraic over $\mathbb{R}$.
\\ \indent Suppose for contradiction that $m>0$, then there exists $0 < \varepsilon < m$ and a polynomial in real coefficients
\[
	g(x) = x^2 - (z_0+\overline{z_0})x + z_0 \overline{z_0} + \varepsilon
\]
with roots $z_1, \overline{z_1} \in \mathbb{C}$. If $z_0 \in \mathbb{R}$, then we have 
\[
	g(x) = x^2 + z_0^2 + \varepsilon \Rightarrow z_1 = \sqrt{z_0^2 + \varepsilon}
\]
If $z_0 \in \mathbb{C} \backslash \mathbb{R}$, then we have that $Re(z_0) = Im(z_1)$. Let's denote it $a$. $z_0 = a+bi$ and $z_1 = a+ci$. Then we have $a^2+ c^2 = a^2 + b^2 + \varepsilon \Rightarrow c = \sqrt{b^2 + \varepsilon}$.
\\ \indent Because
\[
	z_0 \overline{z_0} + \varepsilon = z_1 \overline{z_1} \Rightarrow |z_1| > |z_0| \Rightarrow f(z_1) > m
\]
Let $n \in \mathbb{N}$ be fixed and consider another real polynomial
\[
	G(x) = [g(x) - \varepsilon]^n - (-\varepsilon)^n = \prod_{i=1}^{2n} (x - \alpha_i) = \prod_{i=1}^{2n} (x-\overline{\alpha_i})
\]
then clearly $G(z_1) = 0$, so we may assume that $z_1 = \alpha_1$ without loss of generosity. Now we substitute $\zeta$, then we get
\[
	|G(\zeta)|^2 = \prod_{i=1}^{2n} f(\alpha_i) \geq f(\alpha_1)m^{2n-1}
\]
Furthermore, 
\[
	|G(\zeta)| \leq |g(\zeta) - \varepsilon|^n + |-\varepsilon|^n = |\zeta^2 - (z_0+ \overline{z_0})\zeta +z_0 \overline{z_0}|^n + \varepsilon^n = f(z_0)^n + \varepsilon^n = m^n + \varepsilon^n
\]
then
\[
	f(\alpha_1)m^{2n-1} \leq (m^n + \varepsilon^n)^2
\]
\[
	\frac{f(\alpha_1)}{m} \leq \left ( 1+ \left( \frac{\varepsilon}{m} \right)^n \right)
\]
then as $n \rightarrow \infty$, we get $f(\alpha_1) \leq m$, which contradicts our assumption.
\end{proof}
\begin{prp} Let $\varphi$ be an archimedean valuation of $\mathbb{Q}$, then there exists $\alpha \in \mathbb{R}^{\geq 0}$ such that for all $x \in \mathbb{Q}$,
\[
	\varphi(x) = |x|^\alpha
\]
where $|{\color{white}x}|$ the usual absolute value of $\mathbb{Q}$.
\end{prp}wh 
\begin{proof} Let $m,n \in \mathbb{Z}$ both greater than $1$. For any integer $t>0$, we can write
\[
	m^t = a_0 + a_1n + \cdots +a_s n^s
\]
where $a_i \in \mathbb{Z}$ and $0 \leq a_i < n$ with $a_s \neq 0$. Because $a_s \geq 1$, we have that $n^s \leq m^t$, which taking the log gives $s \log{n} \leq t \log{m}$. We thus have
\[
	s \leq t \dfrac{\log{m}}{\log{n}}
\]
Also we have that $\varphi(n) \leq \varphi(1) + \cdots + \varphi(1) = n$ for any $n$, so we get
\[
	\varphi(a_i) \leq a_i < n
\]
It follows that 
\[
	\begin{array}{lcl}
    	\varphi(m^t) & \leq & \varphi(a_0) + \cdots \varphi(a_s)\varphi(n)^s \\
        & \leq & n ( 1+ \cdots \varphi(n)^s) \\
        & \leq & n (s+1) M^s \text{ where } M = \max(1, \varphi(n)) \\
        & \leq & n \left ( 1+ t \dfrac{\log{m}}{\log{n}} \right )[\max(1, \varphi(n))]^{t(\log{m}/\log{n})}
    \end{array}
\]
Take the $t^{\text{th}}$ root of unity, we get 
\[
	\varphi(m) \leq \max(1, \varphi(n))^{(\log{m}/\log{n})}
\]
Now we assert that $\varphi(n)>1$ for $n>1$ because if there exists $n_0 >1$ with $\varphi(n_0) \leq 1$, then we get
\[
	\varphi(m) \leq 1 \text{ for all } m \in \mathbb{Z}
\]
which contradicts \Cref{archchar} which says that archimedean valuation has $\{ \varphi(m) ~|~ m \in \mathbb{Z} \}$ unbounded.  Hence we achieve a relationship
\[
	\varphi(m) \leq \varphi(n)^{\log{n}/\log{m}} \Rightarrow \varphi(m)^{1/\log{m}} \leq \varphi(n)^{1/\log{n}}
\]
Because $m,n$ were arbitrary, we get the equality
\[
	\varphi(m)^{1/\log{m}} = \varphi(n)^{1/\log{n}} = e^{\alpha}
\]
for some $\alpha \in \mathbb{R}^{> 0}$ because $\varphi(m) > 1 \Rightarrow \varphi(m)^{1/\log{m}} >1$. Hence we get
\[
	\varphi(n) = e^{\alpha \log{n}} = n^\alpha
\]
for any $n >1$, so for any $n \mathbb{Z}$, $\varphi(n) = |n|^\alpha$. By multiplicative property of $|{\color{white}x}|$, we conclude that $\varphi(x) = |x|^\alpha$ for all $x \in \mathbb{Q}$.
\end{proof}
Let $k$ be a number field and $\varphi$ be an archimedean valuation, then we have that $\varphi$ is an extension of $|{\color{white}x}|$, the absolute value of $\mathbb{Q}$. Because $k/\mathbb{Q}$ is an algebraic extension, by the existence theorem, all archimedean valuations arise from $\sigma: k \rightarrow \mathbb{C}$ as $\mathbb{C} = \overline{\mathbb{Q}_{|{\color{white}x}|}}$, i.e. classified by the complex embeddings of $k$, $Hom(k, \mathbb{C})$. We have
\[
	Hom(k, \mathbb{C}):=\{ \sigma_1, \ldots, \sigma_n \}
\]
where $\sigma_1, \ldots, \sigma_r$ is the real embeddings of $k$ and $\sigma_{r+1}, \overline{\sigma_{r+1}}, \ldots, \sigma_{r+s}, \overline{\sigma_{r+s}}$ which are complex embeddings of $k$. Thus we have $r+2s = n$. Because $|\sigma(x)| = \sigma(x)\overline{\sigma(x)} = |\overline{\sigma(x)}|$, we have $\sigma_i$ and $\overline{\sigma_i}$ induces the same topology. Now consider the following set of valuations $M_k$
\[
	\sigma_1, \ldots, \sigma_{r+s}, v_{\mfk{p}}
\]
where $v_\mfk{p}$ is the normalized discrete nonarchimedean valuations of $k$ with $\mfk{p}$ varies over all prime ideals of $k$. $M_k$ is called the \red{canonical set} of $k$. We will use $v \in M_k$ to mean both the archimedean and nonarchimedean valuations of $k$ to unify the notation.
\begin{prp} Let $v_1, v_2 \in M_k$, then $v_1$ and $v_2$ are not equivalent.
\end{prp}
\begin{proof}
	We will prove later in \Cref{><} that for any two distinct $v_i$ with its corresponding valuation $\varphi_i$, there exists $x \in k$ such that
    \begin{equation*} \tag{$\dagger$}
    	\varphi_1(x) > 1 \text{ and } \varphi_2(x) <1
    \end{equation*}
    Because the set $B(0,1)$ consists of $x \in k$ such that $x^n \rightarrow 0$ as $n \rightarrow \infty$, $B(0,1)$ is determined the topology. ($\dagger$) clearly shows that $B(0,1)$ is distinct elements in the canonical set, hence $v_1$ and $v_2$ are not equivalent. This easily extends to the fact that $k_{v_1}$ and $k_{v_2}$ are not isomorphic as their $B(0,1)$ are different when restricted to $k$.
\end{proof}
We conclude this section by observing that $M_k$ is, in fact, a system of representatives of all prime divisors of $k$, hence the name canonical set of $k$.
%%%%%%%%%%%%%%%%%%%%%%%%%%%%%%%%%%%%% new section %%%%%%%%%%%%%%%%%%%%%%%%
\section{Topology on local fields}
Let $\mathcal{O}$ be a Dedekind domain with $\mfk{p}$ be a prime ideal. $\mathcal{O}_\mfk{p}$, the localization of $\mathcal{O}$ at $\mfk{p}$ is a discrete valuation ring. We have $(\pi)$ is the unique maximal ideal of $\mathcal{O}_\mfk{p}$. We define
\[
	A_n:= \mathcal{O}_\mfk{p}/(\pi)^n
\]
which gives a inverse system with canonical epimorphisms. We let
\[
	\mathcal{O}_v^{a} := \varprojlim_m A_m
\]
where the superscript a refers to its algebraic construction. Because for any $m$, we have an isomorphism
\[
	\mathcal{O}/\mfk{p}^m \cong \mathcal{O}_\mfk{p}/(\pi)^m
\]
we get $\mathcal{O}_v^a  = \displaystyle \varprojlim_m \mathcal{O}/\mfk{p}^m$ which is analogous to the construction of $p$-adic numbers from integers. By the definition of the inverse limit, the elements in $\mathcal{O}_v^a$ is of the form $(x_i)$ where $\pi_m(x_m) = x_{m-1}$ with $\pi_m : A_m \rightarrow A_{m-1}$ the canonical epimorphism.
\\ \indent If $|\mathcal{O}_k/\mfk{p}| = p^{f}$ with $(p) = \mfk{p} \cap \mathbb{Z}$, then we can choose a system of representatives of $a_0, \ldots, a_{p^f-1}$ with $0$ included, then for any $(x_i)$ n $\mathcal{O}_v^a$, 
\[
	x_1 \equiv c_0 ~ (\bmod{\pi})
\]
So we consider $x_1 - c_0 = (x_i - c_0)$, then because $x_1 - c_0$ is divisible by $\pi$, the first coordinate is $0$, then because we have the commutative diagram
\begin{center}
	\begin{tikzpicture}
    	\matrix (m)[matrix of math nodes, nodes={anchor=center}, row sep=7em, column sep=2em]{
        	\mathcal{O}_\mfk{p} &  & \\
            \mathcal{O}_\mfk{p}/(\pi)^n & & \mathcal{O}_\mfk{p}/(\pi)^m \\
        };
        \draw[->]
        	(m-1-1) edge node[left]{$\pi_n$} (m-2-1)
            (m-2-1) edge node[below]{$\pi^n_m$} (m-2-3)
            (m-1-1) edge node[right]{$\pi_m$} (m-2-3);
    \end{tikzpicture}
\end{center}
with $\pi^n_m: \mathcal{O}_\mfk{p}/(\pi)^n \rightarrow \mathcal{O}_\mfk{p}/(\pi)^m$ a canonical epimorphism. We see that all $x_i - c_0$ are divisible by $\pi$, hence
\[
	\begin{array}{lcl}
    	x - c_0 & = & (x_i - c_0) \\
        & = & (\pi x_i') \\
        & = & \pi(x_i')
    \end{array}
\]
we may continue the process and realize that
\[
	\begin{array}{lcl}
    	x & = & c_0 + (x-c_0)  \\ 
        & = & c_0 + (\pi x') \\ 
        & = & c_0 + \pi ( c_1 - (x' - c_1)) \\
        & = & c_0 + c_1 \pi + c_2 \pi^2 + \cdots
    \end{array}
\]
where $c_i$ are chosen from the system of representatives. Hence we clearly get that all elements in $\mathcal{O}_v^a$ can be uniquely expressed as a infinite sum of the form
\[
	\sum_{n=0}^\infty c_n \pi^n
\]
which correspond to $x = (x_i)$ with $x_i = \sum_{n=0}^i c_n \pi^n$. 
\\ \indent Now we will construct \emph{the ring of integers} $\mathcal{O}_v$ topologically when $v$ is a nonarchimedean valuation. We may define the following invariants
\begin{enumerate}
	\item $\mathcal{O}_v = \{ x \in k_v ~|~ v(x) \geq 0 \} = \{ x \in k_v ~|~ \varphi_v(x) \leq 1 \}$, the \red{valuation ring}  of $k_v$.
    \item $\mfk{p}_v =  \{ x \in k_v ~|~ v(x) > 0 \} = \{ x \in k_v ~|~ \varphi_v(x) < 1 \}$, the \red{maximal ideal} of $\mathcal{O}_v$
    \item $U_v =  \{ x \in k_v ~|~ v(x) = 0 \} = \{ x \in k_v ~|~ \varphi_v(x) = 1 \}$, the \red{unit group} of $\mathcal{O}_v$.
\end{enumerate}
We observe that $v$ is discrete as $v$ is nonarchimedean exponential valuation of a number field, hence $\mathcal{O}_v$ is a discrete valuation ring with some fixed prime number $\pi \in \mathcal{O}_v$.  
\begin{prp} We have the isomorphism
\[
	\mathcal{O}_v/\mfk{p}_v^n \cong \mathcal{O}_\mfk{p}/\mfk{p}\mathcal{O}_\mfk{p}^n
\]
In particular, $\mathcal{O}_v/\mfk{p}_v$ is finite.
\end{prp}
\begin{proof} 
	We first need to show that $\mathcal{O}_v$ is the closure of $\mathcal{O}_\mfk{p}$. For any Cauchy sequence $\{ x_n \}$ of $\mathcal{O}_v$ that converges to $x$, we have that $\varphi_\mfk{p}(x_n) \leq 1 \Rightarrow \varphi_\mfk{p}(x) \leq 1 \Rightarrow x \in \mathcal{O}_v$, hence $\mathcal{O}_v$ is closed. Now let $x \in \mathcal{O}_v \backslash \mathcal{O}_\mfk{p}$. We have $x = \displaystyle \lim_{n \rightarrow \infty} x_n$ for $x_n \in k$. Because $\varphi_\mfk{p}$ is discrete, we have that there exists $n_0$ such that $n \geq n_0 \Rightarrow \varphi_\mfk{p}(x) = \varphi_\mfk{p}(x_n) \leq 1$. This is same as saying that $x_n = a_n/b_n$ for some $a_n, b_n \in \mathcal{O}_k$ with $\varphi_\mfk{p}(b_n) = 1$. Because we know that $\mfk{p} = (\pi)$ for some prime element, we can always find $y_n$ such that
    \[
    	b_ny_n \equiv a_n ~(\bmod{~\pi^n})
    \]
    Then We have $\varphi_\mfk{p}(x_n - y_n) \leq \dfrac{1}{(p^f)^n}$. $x = \displaystyle \lim_{n \rightarrow \infty} y_n$, hence $x \in \overline{\mathcal{O}_\mfk{p}}$, the closure of $\mathcal{O}_\mfk{p}$. Combining the two, we se e that $\mathcal{O}_v$ is a closure of $\mathcal{O}_\mfk{p}$.
\\ \indent Now consider the map
\[
	a \mapsto a~(\bmod{~\mfk{p}_v^n}) \text{ defined for } \mathcal{O}_\mfk{p} \rightarrow  \mathcal{O}_v/\mfk{p}_v^n
\]
The same prime element in $\mathcal{O}_\mfk{p}$ is also a prime element for $\mathcal{O}_v$ as they are both discrete valuation ring. Thus the kernel is simply $\mfk{p}\mathcal{O}_\mfk{p}^n$, so we now show that the map is surjective. Let $x \in \mathcal{O}_v$, then there exists $a \in \mathcal{O}_\mfk{p}$ such that
\[
	\varphi_\mfk{p}(x - a) \leq \frac{1}{(p^f)^n}
\]
because it either lies inside $\mathcal{O}_\mfk{p}$ or is a limit point, then by definition $x-a \in \mfk{p}^n_v$, i.e. $x \cong a ~(\bmod{~\mfk{p}_v^n})$. The isomorphism has been established.
\end{proof}
Hence we can choose a system of representative of $\mathcal{O}_v/\mfk{p}_v$ from $\mathcal{O}_v$ containing $0$ which will be denoted $\mathcal{R}$ where $p^f = |\mathcal{O}_v/\mfk{p}_v|$ with $(p) = \mfk{p} \cap \mathbb{Z}$ and $f = [\mathcal{O}_k/\mfk{p}:\mathbb{Z}/(p)]$.  
\begin{thm} Every element $x$ in $k_v$ can be uniquely expressed into the form
\[
	x = \sum_{i=m}^\infty a_i \pi^i
\]
where $m \in \mathbb{Z}$ with $c_i \in \mathcal{R}$. Furthermore, every such infinite sum represents an element of $k_v$ with $v(x) = m$.
\end{thm}
\begin{proof}
	Let $x = \pi^m u$ with $u \in \mathcal{O}_v^*$. Let $\mathcal{O}_v/\mfk{p}_v \cong \mathcal{O}_\mfk{p}/\mfk{p}\mathcal{O}_\mfk{p}$, then $\obar{u} \in \mathcal{O}_v/\mfk{p}_v$ has a unique nonzero representation $a_0 \in R$. Hence we have $u =a_0 + \pi b_1$ for some $b_1 \in \mathcal{O}_v$. Assume $a_0, \ldots, a_{n-1} \in \mathcal{R}$ satisfies
    \[
    	u = a_0 + \cdots + a_{n-1}\pi^{n-1} + \pi^n b_n
    \]
    with $a_i$ uniquely determined by $b_n = a_n + \pi b_{n+1}$, then by the same process, we obtain
    \[
    	u=a_0 + \cdots + a_n\pi^n + \pi^{n+1}b_{n+1}
    \]
    Then we have uniquely determined $\sum_{i=m}^\infty a_i \pi^i$ by $u$, and this converges to $u$ because $\pi^{n+1}b_{n+1}$ converges to $0$.
\end{proof}
\noindent We consider the open balls
\[
	B(0, p^{-fs})
\]
where $s \in \mathbb{Z}$, then
\[
	y \in B(0, p^{-fs}) \Leftrightarrow v_\mfk{p}(y) > s \Leftrightarrow y = \displaystyle \sum_{n=s+1}^\infty c_n \pi^n
\]
Hence we immediately see that
\begin{enumerate} 
	\item $\mathcal{O}_v = B(0, p^f)$,
    \item $\mfk{p}_v = B(0, 1)$,
    \item $U_v = \mathcal{O}_v - \mfk{p}_v$.
\end{enumerate}
by comparing the definition. Now for any fundamental open sets $B(0, \gamma)$, by discreteness, we have
\[
	B(0, \gamma) = \overline{B(0, \gamma-\varepsilon)}
\]
for sufficiently small $\varepsilon >0$. Hence all fundamental open sets are closed. It follows that $\mathcal{O}_v, \mfk{p}_v, U_v$ are all open and closed. In fact, for any fundamental open sets $B(x, p^{-fs})$ and $B(0, p^f)$ we have a homeomorphism
\[
	y \mapsto \pi^{-s-1}(y-x)
\]
\begin{thm} $k_v$ is locally compact topological field with compact subsets $\mathcal{O}_v$ and $\mathcal{O}_v^*$.
\end{thm}
\begin{proof}
	The proof that $k_v$ is a topological field coincides with the proof that $\mathbb{R}$ is a topological space in \Cref{Rlocallycompact}. Let's consider 
    \[
    	A_m:= \mathcal{O}_\mfk{p}/\mfk{p}\mathcal{O}_\mfk{p}^m
    \]
then we have that $A_m$ is finite, hence compact, so by the Tychonoff theorem, we have that the product
\[
	\prod_{m=1}^\infty A_m
\]
is compact. Let $(a_m) \notin \mathcal{O}_v$, then there exists $m_0$ such that $\varphi_{m_0}(a_{m_0}) \neq a_{m_0 -1}$ with $\varphi_{m_0} : A_m \rightarrow A_{m-1}$, the canonical epimorphism. Now consider the set
\[
	\{ a_0 \} \times \cdots \times \{ a_{m_0} \} \times \prod_{m>m_0} A_m
\]
is open in the product topology, and is disjoint from $\mathcal{O}_v$ since beginning of every element has $\varphi_{m_0}(a_{m_0}) \neq a_{m_0 - 1}$. This shows that $\mathcal{O}_v^c$ is open, and thus $\mathcal{O}_v$ closed. $\mathcal{O}_v$ is a closed subset of $\prod A_m$, hence $\mathcal{O}_v$ is compact. $\mathcal{O}_v^*$ is a closed subset of $\mathcal{O}_v$, hence compact. 
\end{proof}
We conclude by saying that because any two fundamental open sets are homeomorphic, and fundamental open sets are open, closed, and compact simultaneously.

We have established in the previous section that there was a unique representation of the same form for $\mathcal{O}_v$, so we have established a algebraic isomorphism
\[
	\mathcal{O}_v \cong \mathcal{O}_v^a
\]
We then define $v_\mfk{p}^a : \mathcal{O}_v^a \rightarrow \mathbb{Z} \cup \{ \infty \}$ by $v_\mfk{p}^a(x) = m$, where $m$ is the smallest such that $c_m$ is nonzero. Under the isomorphism above, we see that $v_\mfk{p}^a$ induces an exponential valuation of $\mathcal{O}_v$. This coincide with normalized discrete valuation $v_\mfk{p}$.
\\ \indent For all $\mathcal{O}_\mfk{p}/(\pi)^n$, we give discrete topology and give the product 
\[
	\prod \mathcal{O}_\mfk{p}/(\pi)^n
\]
the product topology. Finally, one gives $\mathcal{O}_v^a$ the subspace topology. We will continue to denote $\mathcal{O}_\mfk{p}/(\pi)^n = A_n$. Then the fundamental open sets in the product topology is of the form,
\[
	\prod_{n=1}^\infty B_n \cap \mathcal{O}_v^a
\]
where $B_n \subset A_n$, and $B_n = A_n$ for almost all $n$. Then because of the finiteness the above set is of the form
\[
	U = \left ( \prod_{n=1}^m B_n \times \prod_{n=m+1}^\infty A_n \right ) \cap \mathcal{O}_v^a
\]
then any $ x \in U$, we have that $x \in B(x, p^{-fm}) \subset U$. In fact, if $y \in B(x, p^{-fm})$, then 
\[
	y - x = \sum_{n=m+1}^\infty c_n \pi^n \Rightarrow y = x + \sum_{n=m+1}^\infty c_n \pi^n
\]
then because $x \in U$ and the $n$-th coefficient where $n>m$ doesn't matter, so we have $y \in U$. 
This shows that the metric topology is finer.
\\ \indent Conversely, if $x \in k_v$ with 
\[
	x = \sum_{n=1}^\infty c_n \pi^n 
\]
then we have that 
\[
	B(x, p^{-fm}) = \prod_{n=1}^m \{ c_n \} \times \prod_{n=m+1}^\infty A_n
\]
hence we have that the metric topology and the topology induced by the inverse limit coincide. We have thus established an algebraic and topological isomorphism
\[
	\begin{array}{lcl}
		\mathcal{O}_v & \cong & \mathcal{O}_v^a \\
        & = & \varprojlim \mathcal{O}_\mfk{p}/(\pi)^n \\
        & \cong & \varprojlim \mathcal{O}_k / \mfk{p}^n
     \end{array}
\]
\section{Hensel's Lemma}
For a field $F$ with a nonarchimedean prime divisor $\mfk{p}$, we denote $\mathcal{O}$ or $\mathcal{O}_{v_\mfk{p}}$ to be its valuation ring and denote $\mfk{p}$ to be the unique prime ideal of $\mathcal{O}$. $\kappa = \mathcal{O}/\mfk{p}$ will always denote the residue class field of $F$ with respect to $\mfk{p}$.
\begin{lem}[Hensel's Lemma] Let $F$ be complete with respect to the nonarchimedean prime divisor $\mfk{p}$, and let $f \in \mathcal{O}[t]$ be primitive, i.e. $\overline{f}  \neq 0$ in $\kappa[t]$. Suppose $\overline{f}$ admits a factorization
\[
	\overline{f} = GH \text{ in } \kappa[t]
\]
with $G,H$ relatively prime polynomials in $\kappa[t]$, then there exists $g,h \in \mathcal{O}[t]$ such that
\begin{enumerate}
	\item $f=gh$,
    \item $\overline{g} = G, \overline{h} = H$,
    \item $deg(g) = deg(G)$.
\end{enumerate}
\end{lem}
\begin{proof}
	To simplify the notation , we will use the following convention $deg(f) = d_f$. It is clear that $d_{\overline{f}} \leq d_f$ and $d_G + d_H = d_{\overline{f}} \leq d_f$, hence $d_G \leq d_f - d_G$. We let $g_1, h_1 \in \mathcal{O}[t]$ with $\overline{g_1} = G$ and $\overline{h_1} =G$, with $d_{g_1} = d_G$ and $d_{h_1} = d_H$. This can be easily done by observing that
    \[
    	G = \overline{a_s}t^s + \cdots + \overline{a_0}
    \]
    can be represented by 
    \[
    	g_1 = a_st^s + \cdots + a_0
    \]
We may assume that $G$ is monic because $\overline{f} = (\overline{a_s}^{-1}G)(\overline{a_s}H)$, so from now on, $G$ and $g_1$ is monic. Because $(G,H)=1$, there exists $\alpha, \beta \in \mathcal{O}[t]$ such that $\overline{\alpha}G + \overline{\beta}H = 1$. This equality and the modulo $\mfk{p}$ factorization of $\overline{f}$ show that $f-g_1h_1, \alpha g_1 + \beta h_1 - 1 \in \mfk{p}[t]$, i.e. all coefficient of these polynomials are all nonzero. Let $v \in \mfk{p}$, then $\overline{v}$ be the canonical extension of $v$ to $F(t)$. We let
\[
	\varepsilon = \min \left ( \overline{v}(f-g_1h_1), \overline{v}(\alpha g_1 + \beta h_1 -1) \right )
\]
If $\varepsilon = \infty$, then we have that $f-g_1h_1$ has all of its coefficient $0$,  which gives the desired $f$ and $g$, so we assume that $0 < \varepsilon < \infty$. Because there are only finite number of coefficients, there exists $\pi \in \mfk{p}$, one of the coefficient, such that $v(\pi) = x$. For all $i \in \mathbb{N}$, we would like to find $g_i, h_i$ such that
\begin{enumerate}
	\item $f \equiv g_i h_i ~(\bmod{~\pi^i})$,
    \item $\overline{g_i} = G$ and $ \overline{h_i} = H$,
    \item $g_i \equiv g_{i-1} ~(\bmod{~\pi^{i-1}})$, $h_i \equiv h_{i-1} ~(\bmod{~\pi^{i-1}})$ for $i \geq 2$,
    \item $d_{g_i} = d_G$ and $d_{h_i} \leq d_f - d_G$.
\end{enumerate}
Let $c_i$ denote the coefficient of $f-g_1h_1$, then $v(c_i) \geq v(\pi) \Rightarrow v(c_i/\pi) = v(c_i) - v(\pi) \geq 0$. This shows that there exists $d_i \in \mathcal{O}$ such that $c_i = \pi d_i \Rightarrow f-g_1 h_1 \in \pi \mathcal{O}[t]$, hence i) is satisfied. ii) and iv) is satisfied by construction. Now suppose that we have established the result for $n-1$ with $n \geq 2$ and also that we have found 
\[
	\begin{array}{lcl}
    	g_n & = & g_{n-1} + \pi^{n-1}u \\
        h_n & = & h_{n-1} + \pi^{n-1}v
    \end{array}
\]
for some $u,v \in \mathcal{O}[t]$, then it is clear that iii) is satisfied. Notice that we have not used the induction hypothesis of iii), so the construction can be used for $i=2$. Also because $\pi \in \mfk{p}$, it is clear that
\[
	\overline{g_n} = \overline{g_{n-1}}+ \overline{\pi^{n-1}u} = \overline{g_{n-1}} = G \text{ and } \overline{ h_n } = \overline{ h_{n-1} } + \overline{\pi^{n-1}v} = \overline{h_{n-1}} = H
\]
Now to show we observe the following equivalence
\[
	\begin{array}{lcl}
    	f \equiv g_n h_n ~(\bmod{~\pi^n}) & \Leftrightarrow & f \equiv g_{n-1}h_{n-1} + \pi^{n-1}(g_{n-1}v + h_{n-1}u) + \pi^{2n-2}uv ~(\bmod{~\pi^n}) \\
        & \Leftrightarrow & f \equiv g_{n-1}h_{n-1} + \pi^{n-1}(g_{n-1}v + h_{n-1}u) ~(\bmod{~\pi^n}) \\
        & \Leftrightarrow & f - g_{n-1}h_{n-1} \equiv \pi^{n-1} (g_{n-1}v + h_{n-1}u) ~(\bmod{~\pi^n}) \\
        & \Leftrightarrow & w:= \dfrac{f-g_{n-1}h_{n-1}}{\pi^{n-1}}  \equiv g_{n-1}v + h_{n-1}u ~(\bmod{~\pi})
    \end{array}
\]
for the last equality comes from the observation that
\[
	x \equiv \pi^{n-1}y ~(\bmod{~\pi^n}) \Leftrightarrow x = \pi^{n-1}y + \pi^nz \Leftrightarrow x/\pi^{n-1} = y + \pi z \Leftrightarrow x/\pi^{n-1} \equiv y ~(\bmod{~\pi})
\]
We now start to actually find $u,v \in \mathcal{O}[t]$. Because $g_1$ is monic, by Euclidean algorithm for polynomials, we have that $w \beta = g_1 q + u$ in $\mathcal{O}[t]$ with $d_u < d_{g_1}$.
\[
	\alpha g_1 + \beta h_1 \equiv 1 \Leftrightarrow w \alpha g_1 + w \beta h_1 \equiv w ~(\bmod{~\pi})
\]
Plugging in $w \alpha g_1 + (g_1 q + u)h_1 \equiv w \Rightarrow (w \alpha + q h_1)g_1 + u h_1 \equiv w ~(\bmod{~\pi})$. Denote $v$ to be the polynomial obtained by replacing all coefficient of $w\alpha + q h_1$ that is divisible by $\pi$ into $0$. We still get the modulo equivalence $\alpha v + \beta u = w$, and gives us i). 
\\ \indent We are left to prove iv). 
\[
	d_u < d_{g_1} \Rightarrow d_{g_n} = deg(g_{n-1}+\pi^{n-1}u) = d_{g_{n-1}} = d_G 
\]
Suppose for contradiction that $d_{h_n} > d_f - d_G$,
\[
	h_n = h_{n-1} + \pi^{n-1}v \text{ and } d_{h_{n-1}} \leq d_f - d_G \Rightarrow d_v > d_f - d_G
\]
\[
	d_{g_{n-1}} = d_G \text{ and } d_{h_{n-1}} \leq d_f - d_G \Rightarrow deg(f- g_{n-1}h_{n-1}) \leq d_f \Rightarrow w \leq d_f
\]
Now we come back to the equality $vg_1 + uh_1 \equiv w ~(\bmod{~\pi})$. Because $deg(vg_1) \geq deg(uh_1)$, we have that the leading coefficient of $vg_1 + uh_1$ is not divisible by $\pi$. However, because $d_w < d_{vg_1}$ we have that
\[
	vg_1 + uh_1 = w + \pi\gamma
\]
for some $\gamma$ with $d_\gamma > w$. The right hand side has leading coefficient divisible by $\pi$ leading into contradiction.
\\ \indent For all $i$, we have $g_i(x) = a_0^{(i)} + \cdots + a_{r-1}^{(i)}t^{r-1} + t^r$ with $r = d_G$. First let $m \geq n \geq N$, then we claim that
\[
	a_j^{(n)} \equiv a_j^{(m)} ~(\bmod{~\pi^N}) 
\]
which can be easily established by induction from $a_j^{(n)} \equiv a_j^{(n-1)} ~(\bmod{~\pi^{n-1}})$. Furthermore, we define $\varphi(x) = e^{-v(x)}$. Let $q=e^{-\varepsilon}$, then 
\[
	\pi^N | x \Leftrightarrow v(x) \geq v(\pi)^N = N\varepsilon \Leftrightarrow e^{-v(x)} \leq e^{-N\varepsilon} \Leftrightarrow \varphi(x) \leq q^N
\]
Thus we have that $\{ a_j^{(i)} \}$ is a Cauchy sequence for fixed $j$, hence converges to $a_j \in \mathcal{O}$ as $F$ is $\mfk{p}$-complete. For $n \geq i$,
\[
	\varphi(a_j^{(i)} - a_j^{(n)}) \leq q^{i} \Rightarrow \varphi(a_j^{(i)} -a_j) \leq q^i \Rightarrow a_j - a_j^{(i)}  \equiv 0 ~(\bmod{~\pi^i})
\]
Clearly, if $a \equiv 0 (~\bmod{~\pi^i}) \Rightarrow \varphi(a) < q^i$ for all $i$, thus $\varphi(a)=0 \Rightarrow a = 0$. We can finally conclude that
\[
	f \equiv g_i h_i \equiv gh ~(\bmod{~\pi^i}) \text{~for all i~} \Rightarrow f=gh
\]
$d_g = d_{g_i} = d_G$ with $\overline{g} = \overline{g_i} = G$ and $\overline{h} = \overline{h_i} =H$ complete our proof.
\end{proof}
\begin{cor} \label{henlem} If $f = a_0 + \cdots + a_nt^n \in F[t]$ is irreducible, then $\overline{v}(f) = \min(v(a_0), \ldots, v(a_n)) = \min(v(a_0), v(a_n))$. In particular, if $f$ is monic and irreducible, i.e. $a_n=1$, then 
\[
	f \in \mathcal{O}[t] \Leftrightarrow a_0 \in \mathcal{O}
\]
\end{cor}
\begin{proof}
	First $f$ primitive has the following equivalences
    \[
    	f \text{ primitive } \Leftrightarrow a_j \notin \mfk{p} \Leftrightarrow v(a_j) = 0 \Leftrightarrow \overline{v}(f) = 0
    \]
    where $a_j$ is some coefficient in $f$. If $\overline{v}(g) = \alpha \neq 0$, then $\overline{v}(g/\alpha) = 0$, hence all polynomial has the form $g = \alpha f$  where $f$ is primitive. If we have proved the result for primitive polynomials, then 
    \[
    	\overline{v}(g) = \overline{v}(\alpha f) = \overline{v}(\alpha) + \overline{f} = \overline{v}(\alpha) + \min(v(a_0), v(a_n)) = \min(v(\alpha a_0), v(\alpha a_n))
    \]
    So we may assume $f$ to be primitive. Now suppose $\min(v(a_0), v(a_n)) > 0$, then there exists $r$ such that $v(a_r) = 0$, $v(a_i) > 0$ for $i \geq r$ and $0 < r < n$. The existence follows from $\overline{v}(f) = 0$. $\overline{f} = \overline{f} \cdot 1 = GH$ modulo $\mfk{p}$. But $d_G = r < n$ gives contradiction. 
    \\ \indent Now further suppose that $f$ is monic, then $\overline{v}(f) = \min(v(1), v(a_0)) = v(a_0)$, so if $v(a_0) \geq 0 \Rightarrow v(a_i) \geq 0$ for all $i$. $\overline{v}(f) \geq 0 \Leftrightarrow f \in \mathcal{O}[t]$ completes the proof. 
\end{proof}
\section{Product Formula}
If $k$ is a number field, $\mathcal{O}_k$ is the algebraic integer of $k$ with $\mfk{p} \subset \mathcal{O}_k$ a prime ideal. Let $\mfk{p} |p$, then we say that $\mfk{p}$ \red{lies over} $p \in \mathbb{Z}$. If the prime number is indexed with $i$, i.e. $\mfk{p}_i$, then we denote its corresponding exponential valuation $v_i$. Also, we note its corresponding valuation by $\varphi_i$.
 \\ \indent Recall that we may define the absolute norm of $N: J_k \rightarrow \mathbb{Q}$ where $J_k$ is the group of all fractional ideals of $k$ by $\mfk{A} \mapsto |\mathcal{O}/\mfk{A}|$. Many classical textbooks show that $N$ is indeed extended from the norm from $k$ to $\mathbb{Q}$ in that if $x \in k$, and $(x)$ is the principal ideal generated by $x$, then 
\[	N((x)) = |N^k_\mathbb{Q}(x)| \]
\begin{prp} Let $\mfk{p}_1, \ldots, \mfk{p}_m$ be lying over $p$, and let $\varphi_i$ and $\varphi$ be the corresponding valuations respectively, then we have
\[
	\varphi(N(x)) = \prod_{i=1}^m \varphi_i(x)
\]
\end{prp}
\begin{proof} By the unique factorization,
\[
	(x) = \prod_{i=1}^m \mfk{p}_i^{n_i} \times \mfk{A}
\]
where $\mfk{A}$ is relatively prime to $\mfk{p}_i$ for all $i$. By definition $v_i(x) = n_i$ where $v_i$ is the corresponding exponential valuation of $\mfk{p}_i$, hence with $f_i =[\mathcal{O}_k/\mfk{p}_i: \zmod{p}]$, we have that $\varphi_i(x) = p^{-f_i n_i}$. Then we have that
\[
	\prod_{i=1}^m \varphi_i(x) = p^{-\sum f_i n_i}
\]
On the other hand, we have that
\[
	N(x) = \prod_{i=1}^m N(\mfk{p}_i)^{n_i} \times N(\mfk{A}) = p^{\sum f_i n_i} \times N(\mfk{A})
\]
because $\mfk{A}$ is the product of power of primes other than $p$, we have that
\[
	\varphi(N(x)) = p^{-\sum f_i n_i}
\]
hence we may conclude the proof.
\end{proof}
\begin{thm} Let $x \in k^*$, then we have
\[
	\prod_{v \in M_k} \varphi_v(x)^{n_v} = 1
\]
where $n_v =1$ when $v$ is nonarchimedean or real archimedean, and $n_v = 2$ when $v$ is complex archimedean.
\end{thm}
\begin{proof} We first let $k=\mathbb{Q}$, then by multiplicity we prove for $x=p$, a prime number. Let $\varphi_q$ be the valuation for prime $q$, then 
\[
	\varphi_q(p) = \left \{ \begin{array}{lcl}
    	1/p & \text{ if } & q = p \\
        1 & \text{ if } & q \neq p
    \end{array} \right .
\]
Then clearly, 
\[
	\prod_{v \in M_\mathbb{Q}} \varphi_v(p)^{n_v} = |p|\varphi_{p}(p) = |p|/p = 1
\]
Now for general number field $k$, we have that
\[ \begin{array}{lcl}
	\displaystyle \prod_{v \in M_k} \varphi_v(x)^{n_v} & = & \displaystyle \prod_{v \in M_k \backslash S_\infty} \varphi_v(x) \prod_{v \in S_\infty} \varphi_v(x)^{n_v} \\
    & = & \displaystyle \prod_{v \in M_k \backslash S_\infty} \varphi_v(x) |N(x)| \\ 
    & = & \displaystyle \prod_{\mfk{p}} \varphi_\mfk{p}(x) |N(x)| \\
    & = & \displaystyle \prod_{p} \prod_{\mfk{p}|p} \varphi_\mfk{p}(x) |N(x)| \\
    & = & \displaystyle \prod_{p} \varphi_{p}(N(x)) |N(x)| \\
    & = & \displaystyle \prod_{v \in M_\mathbb{Q}} \varphi_v(N(x)) = 1
\end{array} \]
\end{proof}
\begin{lem} \label{><} Let $v_1, v_2 \in M_k$ with $k$ a number field be two distinct valuation in $M_k$, then there exists $x \in k$ such that $\varphi_1(x) > 1$ and $\varphi_2(x) < 1$.
\end{lem}
\begin{proof}
	Let $v_1= v_\mfk{p}$ and $v_2 = v_\mfk{q}$ where $\mfk{p}, \mfk{q}$ are two distinct prime ideals of $\mathcal{O}_k$. Consider an element $x \in \mfk{q} \backslash \mfk{p}$ and $y \in \mfk{p} \backslash \mfk{q}$. Then we have $\varphi_\mfk{p}(x) > 1$ and $\varphi_\mfk{q}(x) \leq 1$. If $\varphi_\mfk{q}(x) \neq 1$, then $x$ works. If $\varphi_\mfk{p}(y) \neq 1$, then we have that $y^{-1}$ works. If both $\varphi_\mfk{p}(x) = \varphi_\mfk{q}(y)=1$, then $xy^{-1}$ works.
    \\ \indent Let $v_1 = \sigma$ for some $\sigma \in Hom(k, \mathbb{C})$, and $v_2 = v_\mfk{p}$ then we consider $x \in \mfk{p} \cap \mathbb{Z}$. $\varphi_1(x) > 1$ and $\varphi_2(x) < 1$ as inverse of a rational integer is never integral. If $v_1$ is nonarchimean and $v_2$ is archimedean, then $x^{-1}$ will for for $x \in \mfk{p} \cap \mathbb{Z}$.
    \\ \indent Now for $v_1, v_2$ both archimedean, we know that it arises from different complex embeddings of $k$, we see that there exists $x \in k$ such that $\varphi_1(x) \neq \varphi_2(x)$. Without loss of generosity, suppose that there exists $r$ such that $\varphi_1(x) < r < \varphi_2(x)$, then we see that
    \[
    	\varphi_1(r/x) > 1 \hspace*{1em} \text{ and } \hspace*{1em} \varphi_2(r/x) > 1
    \]
\end{proof}
We have seen the equivalence 
\[
	\varphi(x) < 1 \Leftrightarrow x \in B_\varphi(0,1)
\]
so we have that for two distinct $v_1$ and $v_2$, the open ball centered at $0$ of radius $1$ is different for all $v$, i.e. $v_1$ and $v_2$ are inequivalent.
\begin{lem} Suppose $v_1$ and $v_2$ are two distinct elements in $M_k$. Suppose $x \in k$ satisfies $\varphi_1(x) >1$ and $\varphi_2(x) <1$, then the sequence
\[
	z_m = \frac{x^m}{1+x^m}
\]
converges to $1$ with respect to $v_1$ and converges to $0$ with respect to $v_2$.
\end{lem}
\begin{proof} By performing the same procedure, we get the \emph{usual} theorems of limit, i.e., given $x_n, y_n$ sequences in $k$, and a valuation $\varphi$ of $k$, suppose $x_n \rightarrow x$ and $y_n \rightarrow y$, then
\begin{enumerate}
	\item $\varphi(x_n+y_n) \rightarrow x+y$,
    \item $\varphi(x_n y_n) \rightarrow xy$,
    \item $\varphi(x_n/y_n) \rightarrow x/y$ for $y_n \neq 0$ for all $n$ and $y \neq 0$.
\end{enumerate}
In particular, $\varphi_2(x^m/(1+x^m)) = \varphi_2(x^m)/\varphi_2(1+x^m) \rightarrow 0/1 =0$. Also, we get $\varphi_1(x) > 1 \Rightarrow \varphi_1(1/x) <1 \Rightarrow$
\[
	\varphi_1 \left ( \dfrac{x^m}{1+x^m} \right ) = \varphi_1 \left ( \dfrac{1}{\dfrac{1}{x^m} +1 } \right ) = 1 / \varphi_1 \left ( \dfrac{1}{x^m} +1 \right ) \rightarrow 1/1 = 1 
\]
\end{proof}
\begin{thm}[Approximation theorem] For distinct $v_1, \ldots, v_s \in M_k$, $x_1, \ldots, x_s \in k$, not necessarily distinct, and $\varepsilon > 0$, there exists $x \in k$ such that
\[
	\varphi_i(x - x_i) < \varepsilon
\]
for all $i=1, \ldots, s$. 
\end{thm}
\begin{proof} We first show that there exists $y_i$ such that $\varphi_i(y_i) > 1$ and $\varphi_j(y_i) < 1$ for all $j \neq i$. We have proved the case when $i=2$, so we may proceed by induction. Also, without loss of generosity, we may assume that $i=1$. So suppose by inductive hypothesis we have $y$ such that
\[
	\varphi_1(y) >1 \text{ and } \varphi_i(y) < 1
\]
for $i=1, \ldots, s-1$. If $\varphi_s(y) <1$, then we are done, so first we assume that $\varphi_s(y) = 1$. Let $z \in k$ be such that $\varphi_1(z) > 1$ and $\varphi_s(z) < 1$. Then we can let $m$ large enough so that $y^mz$ works, so we are left with the case where $\varphi_s(y)>1$. We then let
\[
	z_m = \frac{y^m}{1+y^m}
\]
which converges to $1$ with respect to both $v_1$ and $v_2$, and $z_m$ will converge to $0$ with respect to $v_2, \ldots, v_{s-1}$. Pick $z$ such that $\varphi_1(z) >1$ and $\varphi_s(z) <1$, then for sufficiently large $m$, we see that $z_mz$ works.
\\ \indent Now we are given $y_i$ such that 
\[
	\varphi_i(y_i) > 1 \hspace*{2em} \text{ and } \hspace*{2em} \varphi_j(y_i) < 1 \text{ for } j \neq i
\]
Then letting $z_i = y_l/(1+y_l)$ for some appropriate $l$, we see that $z_i$ is arbitrary closed $1$ with respect with $v_i$ and arbitrary close to $0$ with respect with $v_j$ with $j \neq i$. By letting $x=z_1x_1 + \cdots z_sx_s$ with the inequality
\[
	\varphi_i(x - x_i) \leq \varphi_i(z_1)\varphi_i(x_1) + \cdots + \varphi_i(z_i -1) \varphi_i(x_i) + \cdots + \varphi_i(z_s) \varphi_i(x_s)
\]
can be arbitrary small.
\end{proof}

\part{Adeles and Ideles}
\section{Restricted product}
\indent Let $G$ be a topological group, with $*:G \times G \rightarrow G$ denote its continuous multiplication. For any $g \in G$, we define
\[
	i_g: G \rightarrow G \times G \text{ defined by } h \mapsto (g^{-1}, h)
\]
For any open set $U \times V \subset G \times G$, with $U, V$ open in $G$, we have $i_g^{-1}(U \times V) = V$ if $g^{-1} \in U$ and $\varnothing$ otherwise. This shows that $i_g$ is continuous. The composition
\[
	G \xlongrightarrow{~i_g~} G \times G \xlongrightarrow{~*~} G
\]
which sends $h \in G$ to $g^{-1}h$, we see that the preimage of a subset $H$ of $G$ is simply $gH$. Hence for any open set $H \subset G$, we get $gH$ is also open by continuity. 
\\ \indent Let $G_v$ be a locally compact topological groups with $K_v \subset G_v$ be compact subsets which is defined for almost all $v$. We use $M_G$ be the set of indices, and $S_\infty \subset M_G$ be the subset which consists of all indices where $K_v$ is not defined. A \red{restricted product} $\mathbb{I}_G$ of $G_v$ with respect to $K_v$ is a subset of the direct product
\[
	\prod G_v
\]
consisting of $(x_v) \in \prod G_v$ where $x_v \in K_v$ for almost all $v$, i.e. for all but finitely many $v$. We do not give $\mathbb{I}_G$ the subspace topology, but we consider the subset 
\[
	\prod_{v \in S} U_v \times \prod_{v \notin S} K_v
\]
where $U_v \subset G_v$ open sets and $S \subset M_G$ be a finite subset containing $S_\infty$ to be the basis of the topology for $J$. The simplest set of this form is
\[
	\mathbb{I}_S = \underbrace{\prod_{v \in S} G_v}_\text{(1)} \times \underbrace{\prod_{v \notin S} K_v}_\text{(2)}
\]
(1) is locally compact as $S$ is finite and $(2)$ is compact by Tychonoff, so the total product is locally compact. (The space itself is the compact neighborhood of every point in the space.) We easily see that
\[
	\mathbb{I}_G = \bigcup_{\substack{ S_\infty \subset S \\ S \text{ finite }}} \mathbb{I}_S
\]
hence for all $x \in \mathbb{I}_G$, there exists a compact neighborhood $K_i \subset \mathbb{I}_S$ of $x$ for some $\mathbb{I}_S$ containing $x$, as $\mathbb{I}_S$ is locally compact. $\mathbb{I}_G$ is itself locally compact. Because we haven't used any properties of $G_v$ being topological groups, we may generalize the construction to simply collection of $X_v$ locally compact topological spaces with $K_v$ compact subsets with a locally compact restricted product $\mathbb{I}_X$.
\begin{prp} Let $G_v$ be a locally compact set with $K_v$ be a compact subset of $G_v$ defined for almost of $v \in M_G$ with $M_G$ for some indexing set, then the restricted product of $G_v$ with respect to $K_v$ is a topological group. Furthermore, if $G_v$ are topological rings (resp. topological fields), then the restricted product is also a topological ring.
\end{prp}
\begin{proof}
	Let $f$ be an operation whether it be $+$ or $*$, then consider an open set
    \[
    	\left ( \prod U_v \times \prod \mathcal{O}_v \right ) \times \left ( \prod U_v' \times \prod \mathcal{O}_v \right ) = \prod (U_v \times U_v') \times \prod (\mathcal{O}_v \times \mathcal{O}_v)
    \]
    then its preimage will have the form
    \[
    	\prod f_v^{-1}(U_v \times U_v') \times \prod f_v^{-1}(\mathcal{O}_v \times \mathcal{O}_v)
    \]
    where $f = \prod f_v$ with $f_v: G_v \rightarrow G_v \times G_v$, the operation function for each component. We know that the preimages are open as $k_v$ are all topological rings. The same reasoning works for negation.
\end{proof}
We have shown that for all $v \in M_k$ where $k$ is a number field, $k_v$ is locally compact. For all nonarchimdean prime divisors $v$, we have a compact subset $\mathcal{O}_v$ of $k_v$. The restricted product of $k_v$ with respect to $\mathcal{O}_v$ will be called the \red{adeles} of $k$, denoted $\mathbb{A}_k$. If we have compact subsets $\mathcal{O}_v^*$ with $k_v^*$, then the restricted product will be called the \red{ideles} of $k$ and will be denoted $\mathbb{I}_k$.
\section{Finiteness and Unit Theorem}
Let $\varphi: \mathbb{I}_k \rightarrow \mathbb{R}^{>0}$ be a map defined by
\[	
	x \mapsto \prod_v \varphi_v(x_v)
\]
then the map is clearly an epimorphism. Because for any element, we can adjust one of the archimedean valuations. The next step is to show that the map is continuous. We first prove a lemma that the product of positive reals slightly changes if each term changes slightly.
Let $(b_1,b_2)$ be an interval in $\mathbb{R}$. For any $x \in \mathbb{I}_k$ such that $\varphi(x) \in (a,b)$ and let $S$ be the finite subset of $M_k$ containing all valuations such that $\varphi_v(x_v) \neq 1$
then we have
\[
	\prod_{v \in M_k} \varphi_v(x_v) = \prod_{v \in S} \varphi_v(x_v)
\]
then by \Cref{approxprod}, there exists $\varepsilon \in \mathbb{R}$ such that
\[
	b_1 < \prod_{v \in S} (\varphi_v(x_v) - \varepsilon) \text{ and } \prod_{v \in S} (\varphi_v(x_v) + \varepsilon) < b_2
\]
Now consider the open set
\[
	U = \prod_{v \in S} B(x_v, \varepsilon) \times \prod_{v \notin S} \mathcal{O}_k^*
\]
Clearly $x \in U$, and for any $y \in U$,
\[
	\varphi(y) = \prod_{v \in S} \varphi_v(x_v)  = \prod_{v \in S} \varphi_v(y_v - x_v + x_v) \leq \prod_{v \in S} \varphi_v(x_v) + \varphi_v(x_v - y_v) \leq \prod_{v \in S} \varphi_v(x_v) + \varepsilon < b_2
\]
The reverse triangle inequality implies that
\begin{multline*}
	| \varphi_v(x_v) - \varphi_v(y_v) | \leq \varphi_v(x_v - y_v) \Rightarrow  \varphi_v(x_v) - \varphi_v(y_v) \leq \varphi_v(x_v - y_v) \\ \Rightarrow \varphi_v(x_v) \leq \varphi_v(x_v - y_v) + \varphi_v(y_v) \Rightarrow \varphi_v(x_v) - \varphi_v(x_v - y_v) \leq \varphi_v(y_v)
\end{multline*}
\[
	\varphi(y) = \prod_{v \in S} \varphi_v(y_v) \geq \prod_{v \in S} \varphi_v(x_v) - \varphi_v(x_v - y_v) \geq \prod_{v \in S} \varphi_v(x_v) - \varepsilon > b_1  
\]
which shows that $U$ is a subset of the preimage of $(a,b)$. We, thus, have shown that the preimage is open and that the map is continuous. Clearly, $\{1\} \subset \mathbb{R}^{>0}$ is closed, the $ker(\varphi)$ is closed and will be denoted $\mathbb{I}_k^0$.
\\ \indent The second map to consider is a map from the ideles to the fractional ideals of $k^*$ denoted, $J_k$,
\[
	\psi: \mathbb{I}_k \rightarrow J_k
\]
then the map is defined by
\[
	x \mapsto \prod_{v_\mfk{p} \in M_k \backslash S_\infty } \mfk{p}^{v_\mfk{p}(x_{v_\mfk{p}})}
\]
For any fractional ideal
\[
	\prod_{\mfk{p}} \mfk{p}^{e_\mfk{p}}
\]
we have an element $x=(x_v)$ in the preimage by setting $x_v = e_{\mfk{p}}$ if $v=v_{\mfk{p}}$ and $x_v = 1$ otherwise. This show that the map is surjective. The kernel of this map consists of elements such that $\varphi_v(x_v)= 1$ for all $v \in S_\infty$. For $x \in k^*$, we have that $\psi(x) = (x)$ where the latter is the principal ideal generated by $x$. If we denote $P_k$ to be the principal ideals of $k*$, then we get an induced surjective map
\[
	\mathbb{I}_k/k^* \rightarrow J_k/P_k
\]
with $\mathbb{I} \rightarrow J_k/P_k$. The kernel of the $\psi$ is clearly $\mathbb{I}_{S_\infty}$.
\begin{clm} 
We get an isomorphism
\[
	\mathbb{I}_k/k^*\mathbb{I}_{S_\infty} \cong J_k/P_k = Cl_k
\]
where $Cl_k$ is the ideal class group.
\end{clm}
\begin{proof}
	Let $\alpha x$ be such that $\alpha \in k^*$ and $x \in \mathbb{I}_{S_\infty}$, then consider the ideal $(\alpha x)$. Because $(\alpha x) = (\alpha) (x)$, in the image, it is equivalent to $(x)$. $x \in \mathbb{I}_{S_\infty}$ implies that $(x)$ is, in fact, $\mathcal{O}_k$. Now conversely, suppose $ x \in \mathbb{I}_k$ be such that $(x) \in P_k$, then there exists $\alpha \in k^*$ such that $(x) = (\alpha)$. We have that $(x \alpha^{-1}) = \mathcal{O}_k$. We know that $x \alpha^{-1} \in \mathbb{I}_{S_\infty}$, hence $x \alpha^{-1} = y$ for some $y \in \mathbb{I}_{S_\infty}$. We conclude that $x = \alpha y$, hence the kernel is $k^* \mathbb{I}_{S_\infty}$. 
\end{proof}
Similar to the ideal class group, we have  $C_k = \mathbb{I}_k/k^*$ and call it the \red{idele class group}. The relationship between the two class groups can be drawn from the epimorphism
\[
	C_k \rightarrow Cl_k
\]
For any finite set $S$ containing $S_\infty$, we denote $k_S = \mathbb{I}_S \cap k^*$ called the \red{$S$-units} of $k$. $C_k$ is generalized to
\[
	C_S = \mathbb{I}_S/k_S
\]
We will call $C_S$ the \red{$S$-idele class group}.
\begin{clm} The kernel of $C_k \rightarrow Cl_k$ is $C_{S_\infty}$. \end{clm} 
\begin{proof} The kernel is the image of $k^* \mathbb{I}_{S_\infty}$ under $\mathbb{I}_k/k^*$, then $k^*$ cancels, and we are left with $\mathbb{I}_{S_\infty}$. This shows that the kernel is $\mathbb{I}_{S_\infty}/k_{S_\infty}$ which is $C_{S_\infty}$.
\end{proof}
We restrict the domain of the function by considering $\mathbb{I}_k^0$. We denote
\[
	\mathbb{I}^0_S = \mathbb{I}_S \cap \mathbb{I}_K^0
\]
\[
	C_S^0 = \mathbb{I}^0_S/k_S
\]
\begin{clm} The map is still surjective when we shrink the domain from $C_k$ to $C_k^0$ with its kernel $C^0_{S_\infty}$
\end{clm}
\begin{proof} The first goal is to prove that $C_k^0 \rightarrow Cl_k$ is surjective. Let $\mfk{A}$ be a fractional ideal of $k$, then by the surjectivity of the original map, there exists $x \in \mathbb{I}_k$ such that $(x) = \mfk{A}$. Consider $\phi(x)$, then we can change $x$ by one of its archimedean properties and get $x'$ so that $\phi(x') =1$. $C_k = \mathbb{I}_k/k^*$ implies that $x' \in C_k^0$.
\\ \indent As we can view $C_k^0 \subset C_k$, the kernel is clearly $\mathbb{I}_{S_\infty}^0/k_{S_\infty}$ which is $C_{S_\infty}^0$.
\end{proof}
The claims above combine to establish the isomorphism
\[
	C^0_k/\mathcal{C}^0_{S_\infty} \cong C_k
\]
Before we prove the topological property of $C^0_k$, consider the following notation, for any $x \in \mathbb{I}_k$,
\[
	L(x):= \{ \alpha \in k^* ~|~ |\alpha|_v \leq |x_v|_v \text{ for all } v \in M_k \}
\]
Let $\lambda(x) = |L(x)|$, the order, then what we want to show is that if $|x|$ is sufficiently large, then $\lambda(x) > 0$. Observe that $L(x)$ has a canonical bijection with $L(\alpha x)$ for any $\alpha \in k^*$ by
\[
	y \mapsto \alpha y
\]
\begin{prp}[Stronger Approximation Theorem] Let $k$ be a number field, then there exists $c_k$ which is depended only on $k$ such that for all $x \in \mathbb{I}_k$,
\[
	\lambda(x) \geq c_k\varphi(x)
\]
\end{prp}
\begin{proof}
	Let $n = [k:\mathbb{Q}]$ and $w_1, \ldots, w_n$ be the integral basis of $\mathcal{O}_k$, then define
    \[
    	c:= \sup_{v \in S_\infty, i} \left ( \varphi_v(w_i) \right )
    \]
Then the above immediately implies that
\[
	\varphi_v(w_1) + \cdots + \varphi_v(w_n) \leq c
\]
Let $a_v = c/\varphi_v(x_v)$, then $a_v < 2a_v$, so there exists $r_v \in \mathbb{Q}$ and $\varepsilon \in \mathbb{R}^{>0}$ such that $(r-\varepsilon, r + \varepsilon) \subset (a_v, 2a_v)$, then by the aaproximation theorem, there exists $\alpha \in k^*$ such that
\[
	\varphi_v(\alpha - r_v) < \varepsilon
\]
for all $v \in S_\infty$, then 
\[
	|\varphi_v(\alpha) - \varphi_v(r_v)| \leq \varphi_v(\alpha - r_v) < \varepsilon
\]
we get
\[
	\varphi_v(r_v) - \varepsilon < \varphi_v(\alpha) < \varphi_v(r_v) + \varepsilon 
\]
which implies
\[
	\dfrac{c}{\varphi_v(x_v)} < \varphi_v(\alpha) < \dfrac{2c}{\varphi_v(x_v)}
\]
Now for $\alpha x$, consider all finite $v$ such that $\varphi_v(\alpha x_v) > 1$. By letting $m$ be divisible \emph{highly} by all $v$, we see that $\varphi_v(m \alpha x) \leq 1$ for all nonarchimedean $v$, i.e. for all $v \in M_k \backslash S_\infty$. Let $z = m \alpha x$, then we have 
\[
	mc < \varphi_v(z_v) < 2mc \text{ for } v \in S_\infty \text{ and for some } m \in \mathbb{Z},
\]
\[
	\varphi_v(z_v) \leq 1 \text{ for } v \in M_k \backslash S_\infty
\]
By the second inequality, we have that $(z)$, the fractional ideal corresponding to $z$ is an integral ideal of $k$. Hence we have that $N(z) = |\mathcal{O}_k/(z)|$. Let $\Gamma:=\{ y \in \mathcal{O}_k ~|~ y = a_1 w_1 + \cdots + a_n w_n, 0 \leq a_i \leq m \}$. Then $\Gamma$ has $(m+1)^n$ elements. By the pigeonhole principal, we see that there exists an equivalence class in $\mathcal{O}_k/(z)$ with more than $m^n/N(z)$ elements. In fact, if all equivalence classes have less that $m^n/N(z)$, we have $|\Gamma| \leq m^n/N(z) \cdot N(z) = m^n < (m+1)^n = |\Gamma|$. We thus have a $\Gamma' \subset \Gamma$ with more than $m^n/N(z)$ elements which fall into the same equivalence class. Let $y_1 \in \Gamma'$ be fixed, and we vary $y_2 \in \Gamma'$. 
\\ \indent If $y_1 = \sum a_iw_i$ and $y_2 = \sum b_i w_i$,
\begin{align*}
	\varphi_v(y_1 - y_2) & = \varphi_v \left ( \sum a_i w_i - \sum b_i w_i \right ) \\
    & = \varphi_v \left ( \sum (a_i - b_i) w_i \right ) \\
    & \leq \sum \varphi_v(a_i - b_i) \varphi_v(w_i) \\
    & \leq m \sum \varphi_v(w_i) \leq mc \leq \varphi_v(z_v)
\end{align*}
for all $v$ archimedean. We had that $-m \leq a_i - b_i \leq m$. Because $y_1 - y_2, z_v \in \mathcal{O}_k$, we see that $y_1 - y_2 \in (z) \Rightarrow \varphi_v(y_1 - y_2) \leq \varphi_v(z_v)$ for nonarchimedean $v$. This is because $y_1 - y_2$ is more divisible by $\mfk{p}$ then $z_v$. $y_1 - y_2 \in L(z)$, so 
\[
	m^n/N(z) \leq |\Gamma'| \leq \lambda(z) \Rightarrow 1/N(z) \leq \lambda(z) m^{-n}
\]
For any $z \in \mathbb{I}_k$, we have that
\[
	\prod_{v \in M_k \backslash S_\infty} \varphi_v(z_v) = \prod_{i=1}^n \prod_{\mfk{p}|p_i} \varphi_{\mfk{p}}(z_\mfk{p}) = \prod_{i=1}^n p_i^{-\sum f_\mfk{p}n_\mfk{p}} = 1/N(z)
\]
Hence
\[
	\varphi(z) = \prod_{v \in M_k \backslash S_\infty} \varphi_v(z_v) \prod_{v \in S_\infty} \varphi_v(z_v) = 1/N(z) \prod_{v \in S_\infty} \varphi_v(z_v) \leq \lambda(z) m^{-n} (2mc)^n = \lambda(z) 2^nc^n
\]
We conclude that
\[
	\lambda = (2^nc^n)^{-1} \varphi(z)
\]
Because $m \alpha \in k^*$, we see that $\varphi(z) = \varphi(x)$ by the product formula and $\lambda(x) = \lambda(z)$, hence letting $c_k = (2^nc^n)^{-1}$
\end{proof}
\begin{lem} Let $x \in \mathbb{I}_k$ with $\varphi(x) \geq 2/c_k$, then there exists $\alpha \in k^*$ such that
\[
	1 \leq \varphi_v(\alpha x_v) \leq \varphi(x)
\]
for all $v \in M_k$.
\end{lem}
\begin{proof} We have $2 \leq c_k \varphi(x) \leq \lambda(x)$, then there exists $\beta \in \lambda(x)$ nonzero such that $\varphi_v(\beta) \leq \varphi_v(x_v)$ for all $v \in M_k$. Consider $\alpha = \beta^{-1}$, then 
\[
	 1 \leq \varphi_v(\alpha x_v)
\]
and 
\[
	\varphi_v(\alpha x_v) = \dfrac{ \displaystyle \prod_{w} \varphi_{w}(\alpha x_w)}{\displaystyle \prod_{w \neq v} \varphi_w(\alpha x_w)} \leq \dfrac{\varphi(\alpha x)}{1} = \varphi(\alpha x) = \varphi(x)
\]
In fact,
\[
	1 \leq \varphi_w( \alpha x_w) \Rightarrow 1 \leq \prod_{w \neq v} \varphi_w(\alpha x_w) 
\]
\end{proof}
\begin{prp} $\mathcal{C}^0_k$ is compact
\end{prp}
\begin{proof} By the product formula, $\varphi(x)=1$, so we get a well-defined map
\[
	\varphi: C_k \rightarrow \mathbb{R}^{\geq 0}
\]
with the kernel $C^0_k$, because it is the image of $\mathbb{I}_k^0$ under $\mathbb{I}_k/k^*$. First we show that for any $\rho \in \mathbb{R}^{\geq 0}$, $\varphi^{-1}(\rho)$ is homeomorphic to $C^0$. Let $x_\rho \in \mathbb{I}_k$ be such that $\varphi(x_\rho) = \rho$. We can find such by modifying it by one of its archimedean components. Then we see that $\varphi^{-1}(\rho) = x_\rho C^0_k$. In fact,
\[
	\varphi(x_1) = \varphi(x_2) \Rightarrow x_1x_2^{-1} \in ker(\varphi) = C_k^0 \Rightarrow x_1 \in x_2C_k^0
\]
thus homeomorphic as $\mathbb{I}_k$ is a topological group. 
\\ \indent Choose $\rho > 2/c_k$ and pick $x \in \varphi^{-1}(\rho)$, then by the stronger approximation theorem, we get $\alpha_x \in k^*$ such that
\[
	1 \leq \varphi_v(\alpha_x x_v) \leq \rho
\]
for all $v \in M_k$. We have that $v_\mfk{p}$ takes no values between $1$ and $p^f$ where $(p) = \mfk{p} \cap \mathbb{Z}$ and $f = [\mathcal{O}_k/\mfk{p}: \mathbb{Z}/p\mathbb{Z}]$, and there are only finitely many primes $\mfk{p}$ such that $p^f \leq \rho$, hence we have that
\[
	1 \leq |\alpha_x x_v|_v \leq \rho \text{ for } v \in S
\]
\[
	|\alpha_x x_v|_v = 1 \text{ for } v \notin S
\]
for some $S_\infty \subset S \subset M_k$ finite subset. We define 
\[
	T = \prod_{v \in S} \left ( \overline{B(0, \rho)} - B(0, 1) \right ) \times \prod_{v \notin S} \mathcal{O}_v^*
\]
then as each term is compact, we have $T$ compact by Tychonoff. $\varphi^{-1}(\rho)$ is contained inside the image of $T$ as we can change the elements in $\varphi^{-1}(\rho)$ by elements in $k^*$. The image of $T$ is compact and $\varphi^{-1}(\rho)$ is a closed subset of compact space $T$. $\varphi^{-1}(\rho)$ is compact.
\end{proof}
\begin{thm} The class group $Cl_k$ for a number field $k$ is finite.
\end{thm}
\begin{proof} $C_{S_\infty}^0$ is an open subset of $C^0_k$, hence $C^0_k/C_{S_\infty}^0$ is discrete, and compact as $C_k^0$ is compact. $\{ \{ x \} \}_{x \in X}$ is an open covering where $X = C^0_k/C_{S_\infty}^0$. By compactness, there exists a finite subcover, i.e. $C^0_k/C_{S_\infty}^0$ is finite. The isomorphism 
\[
	C_k^0/C_{S_\infty}^0 \cong Cl_k
\]
implies that $Cl_k$ is finite.
\end{proof}
\begin{thm}[Dirichlet $S$-Unit Theorem] For any finite set $S \subset M_k$, $k_S$ has rank $s-1$ where $s=|S|$.
\end{thm}
\begin{proof} Consider the $\log$ map on ideles
\[
	\log : \mathbb{I}_S \rightarrow \mathbb{R}^s
\]
defined by
\[
	\log(x) = (\log \varphi_1(x_1)^{n_1}, \ldots, \log \varphi_s(x_s)^{n_s})
\]
where $\varphi_i$ corresponds to $v_i \in S$, and $x_i$ is the $v_i$-component of $x \in \mathbb{I}_S$. Also, $n_i$ is $1$ or $2$ depending whether $\varphi_i$ is real or complex embedding, respectively. As each component is a composition of continuous maps, they are continuous. Consequently, $\log$ is continuous.
\\ \indent By definition $x \in \mathbb{I}_S^0$ implies that $\varphi(x) =1$ and $\varphi_v(x_v) = 1$ for $v \notin S$. It follows that $\log(\mathbb{I}_S^0)$ lies inside the hyperplane $H$ which is composed of elements $(y_1, \ldots, y_s) \in \mathbb{R}^s$ such that
\[
	y_1 + \cdots + y_s = 0
\]
We may, without loss of generosity, assume that $v_s$ is archimedean. The following $n-1$ vectors
\begin{center}
$\begin{array}{c}
	(a_1, 0, \ldots, 0, -a_1) \\
    (0, a_2, \ldots, 0, -a_2) \\
    \vdots \\
    (0, 0, \ldots, a_{s-1}, -a_{s-1})
\end{array}$
\end{center}
are linearly independent for $a_i$ nonzero. We conclude that $dim_\mathbb{R}(H) = s-1$. Because all vectors lies inside $\log(\mathbb{I}_S^0)$, we have just proved that it generates $H$. 
\begin{clm} $\log(k_S)$ is discrete in $\mathbb{R}^s$.
\end{clm}
\begin{proof} Let $T$ be a bounded subset of $\mathbb{R}^n$, then the value of archimedean valuation on $\log(k_S) \cap T$ is bounded. If the coefficient of polynomials over $\mathbb{Z}$ that satisfies an element in $k^*$ is not bounded, then we need a $x \in k$ with a bigger valuation to make $f(x) = 0$. The coefficients of $f$ is thus bounded. Because the coefficients are in $\mathbb{Z}$ and the degree is bounded by $[k:\mathbb{Q}]$, we have only finitely many such polynomials. Note that we don't have to assume that the polynomial is monic, hence all elements in $k$ satisfies some polynomial with coefficient in $\mathbb{Z}$. $T \cap \log(k_S)$ is, thus, finite. 
\end{proof}
\begin{clm} Let $X$ be a discrete subset of $\mathbb{R}^s$, then it is free abelian of rank $dim(\mathbb{R}X)$ where $\mathbb{R}X$ is the $\mathbb{R}$-span of $X$.
\end{clm}
\begin{proof} We use induction on the dimension of $\mathbb{R}X$. If $dim(\mathbb{R}X)=1$, then $X$ discrete implies that there exists nonzero $x \in X$ such that $x$ is closest to $0$. Now suppose $X$ is generated by $y$ which exists as $dim(\mathbb{R}X) = 1$. $y = kx$ for some $k \in \mathbb{R}$. There exists $t \in \mathbb{Z}$ such that $|k-t| < 1$. Then $ty - x$ is closer to $0$ than $x$ which contradicts the minimality of $x$.  
\\ \indent We now assume that $dim(\mathbb{R}X) = m > 1$ with basis $x_1, \ldots, x_m$ of $\mathbb{R}X$. Let $X_0$ be the subgroup of $X$ generated by $x_1, \ldots, x_{m-1}$, then by the induction hypothesis, we have
\[
	X_0 = \mathbb{Z}x_1 + \ldots + \mathbb{Z}x_{m-1}
\]
Also, let $X'$ be the subset of $X$ consisting of elements of the form
\[
	a_1x_1 + \ldots + a_mx_m
\]
where $0 \leq a_i < 1$ for $i=1, \ldots, m-1$ and $0 \leq a_m \leq 1$ with $a_i \in \mathbb{R}$. The constraints make $X'$ into a bounded subset of a discrete set $X$, hence it is finite. We choose $x' \in X'$ with
\[
	x' = a_1'x_1 + \cdots + a_m' x_m
\]
such that the coefficient of $x_m$ is minimal, i.e. $|a_m'|$ is minimal. Choose an arbitrary element $x \in X$, then there exists $t \in \mathbb{Z}$ such that $x-tx'$ has its $m$-th coefficient $a_m$ satisfy $0 \leq a_m < a_m'$. There exists $x_0 \in X_0$ such that $x - kx' - x_0 \in X'$. Because $|a_m'|$ was minimal, we have that $a_m = 0$, hence $x - kx' - x_0 \in X_0$. By the restriction $0 \leq a_i <1$, we see that $x - kx' - x_0 = 0$. We have already established linear independence, so we have shown that
\[
	X = \mathbb{Z}x_1 \oplus \cdots \oplus \mathbb{Z}x_n
\]
as desired.
\end{proof}
Let $W$ be the $\mathbb{R}$-span of $\log(k_S)$. Because $\log(k_S) \subset W$, there exists an induced epimorphism
\[
	\log : \mathbb{I}_S^0/k_S = C_S^0 \rightarrow H/W
\]
We have proved that $\mathbb{I}_k^0$ is closed, so taking the subspace topology and the quotient topology, we see that $C_S^0$ is a closed subset of $C_k^0$. This shows that $C_S^0$ is compact. As $\log$ is continuous, $H/W$ is a compact subset of $\mathbb{R}^n$ for some $n$.
\\ \indent Suppose that $X$ is a compact subset of $\mathbb{R}^n$. $X$ is then of the form
\[
	X = X_1 \times \cdots \times X_n
\]
As projection maps are continuous by definition, $X_i$ are compact. $X_i$ is a subgroup of $\mathbb{R}$, hence it is either discrete or dense in $\mathbb{R}$. $X_i$ discrete implies that $X_i$ is discrete and compact, i.e. finite. The only finite subgroup of $\mathbb{R}$ is $0$. If $X_i$ is dense in $\mathbb{R}$, then $X_i$ compact implies that $X_i$ closed. $X_i = \mathbb{R}$, i.e. we have that $X$ is of the form $\mathbb{R}^l$ for some $l$. $\mathbb{R}^l$ is compact if and only if $l=0$. $H/W = 0 \Rightarrow H=W$.
\\ \indent We have proved that $\log(k_S)$ has rank $s-1$. The map $\log: k_S \rightarrow \log(k_S)$. Then the kernel is $x \in k$ such that $\varphi_v(x) = 1$ for $v$ archimedean valuation. The image of the kernel is thus bounded, so by similiar argument, the kernel is finite. Hence $k_S$ and $\log(k_S)$ have the same rank.
\end{proof}

\appendix
\addcontentsline{toc}{part}{Appendices}
\section{Lemmas}
\begin{prp}[Reverse Triangle Inequality] \label{reversetri} Let $\varphi$ be a valuation of a field $F$ with $||\varphi|| \leq 2$, then we have
\[
	|\varphi(x) -\varphi(y)| \leq \varphi(x-y)
\]
where $|{\color{white}x}|$ is the usual absolute value of $\mathbb{R}$.
\end{prp}
\begin{proof} For any $x,y \in F$,
\[
	\varphi(y) = \varphi(y-x+x) \leq \varphi(y-x) + \varphi(x) 
\]
which implies
\[
   \varphi(y) - \varphi(x) \leq \varphi(x-y) 
\]
then by symmetry we get
\[
	| \varphi(x)-\varphi(y)| \leq \varphi(x-y)
\]
\end{proof}
\begin{lem} \label{approxprod} Let $a_1, \ldots, a_n, b \in \mathbb{R}^{\geq 0}$ with $b \in (b_1, b_2)$, then there exists $\varepsilon \in \mathbb{R}^{>0}$ such that if
\[
	\prod_{i=1}^n a_i = b
\]
then
\[
	b_1 < \prod_{i=1}^n (a_i - \varepsilon) \text{ \hspace*{1em} and \hspace*{1em} } \prod_{i=1}^n (a_i + \varepsilon) < b_2 
\]
\end{lem}
\begin{proof} Let $0 < \gamma < \min(|b_1 - b|, |b_2 - b|)$ and define $\varepsilon = \min(l, \gamma)$ with $l$ to be defined later then 
\[
	 \begin{array}{lcl}
    	 \displaystyle \prod_{i=1}^n (a_i + \varepsilon) & = & \displaystyle \prod_{i=1}^n a_i + \varepsilon \sum_{j=1}^m \varepsilon^{e_{\varepsilon, j}} a_1^{e_{1,j}} \cdots a_n^{e_{n,j}} \\
        & \leq & b + \varepsilon l \\
        & & (l = \displaystyle \sum_{i=1}^m a_1^{e_{1,j}} \cdots a_n^{e_{n,j}}) \\ 
        & \leq & b + \displaystyle \frac{\gamma}{l} l \\ 
        & \leq & b + \gamma < b_2
    \end{array}
\]
Similarly, we have that 
\[
	\begin{array}{lcl}
	\displaystyle \prod_{i=1}^n (a_i - \varepsilon) & \geq & \displaystyle \prod_{i=1}^n a_i + \varepsilon \sum_{i=1}^m (-\varepsilon)^{e_{\varepsilon, j}} a_1^{e_{j,1}} \cdots a_n^{e_{j,n}} \\
    & \geq & a_i - \varepsilon l \\ 
    & \geq & a_i - \displaystyle \frac{\gamma}{l} l \\
    & \geq & a_i - \gamma > b_1
    \end{array}
\]
which concludes the proof.
\end{proof}
\begin{prp} \label{normlem} Let $K/F$ be finite and $x \in K$. $T_x : K \rightarrow K$ be a $F$-vector space homomorphism defined by $y \mapsto xy$. If $f$ is the characteristic polynomial of $T_x$ and $g$ be the minimal polynomial of $x$ over $F$, then $f=g^d$ where $d=[K:F(x)]$.
\end{prp}
\begin{proof} We have a simple tower
\[
	K/F(x)/F
\]
let $m = [F(x):F]$, then $1, x, \ldots, x^{m-1}$ is its basis. Also let $w_1, \ldots, w_d$ be the $F(x)$-basis of $K$. This gives us 
\[
	w_1, w_1x, \ldots, w_1x^{m-1}; \cdots ; w_d, w_d x, \ldots, w_d x^{m-1}
\]
a $F$-basis of $K$. Let $g(t) = t^m + c_1 t^{m-1} + \cdots + c_m$. Then the matrix representation of $T_x$ with respect to such basis is a block matrix on the diagonal with each block
\[
	\begin{pmatrix}
    	0 & 1 & 0 & \cdots & 0 \\
        0 & 0& 1 & \cdots & 0 \\
        0 & 0 & 0 & \ddots & 1 \\
        -c_m & -c_{m-1} & -c_{m-2} & \cdots & -c_1
    \end{pmatrix}
\]
the characteristic polynomial is then
\[
	\begin{vmatrix}
    	t & -1 & 0 & \cdots & 0 \\
        0 & t & -1 & \cdots & 0 \\
        0 & 0 & t & \ddots & 1 \\
        c_m & c_{m-1} & c_{m-2} & \cdots & c_1
    \end{vmatrix} = 
    t 
	\begin{vmatrix}
    	t & -1 & \cdots & 0 \\
        0 & t & \cdots & \vdots \\
        \vdots & \vdots & \ddots & -1 \\
        c_m & c_{m-2} & \cdots & c_1
    \end{vmatrix}
    - (-1)
    \begin{vmatrix}
    	0 & -1 & \cdots & 0 \\
        0 & t & \cdots & \vdots \\
        \vdots & \vdots & \ddots & -1 \\
        c_{m-1} & c_{m-2} & \cdots & c_1
    \end{vmatrix}
\]
The left determinant is equal to $(t^{m-1} + \cdots + c_{m-1})$ by induction, and the right matrix is also by induction simply $c_m$. The above expression is equal to 
\[
	t(t^{m-1} + \cdots c_{m-1}) + c_m = t^m + c_1t^{m-1} + \cdots + c_m
\]
Thus have thus shown that $f = g^d$.
\end{proof}
In particular, if $x \in F$, then the minimal polynomial is $t-x$, hence the characteristic polynomial is $(t-x)^n$ where $n=[K:F]$ as $F(x) =F$. Then the norm which is the determinant is $x^n$, i.e. $N^K_F(x) = x^n$ for $x \in F$.
\newpage
\section{Quotes}
\begin{aquote}{Don Blasius}
Ideles and adeles are a basic innovation. In addition to making class-field theory easier to formulate, they have interesting new properties themselves, and are essentially to understanding L-functions, Galois representations, and automorphic forms. A huge amount of current research concerns these topics and their interconnections, which taken together go by the name ``non-abelian class field theory" but really have broader interest.
\end{aquote}
\vspace{5pt}
\begin{aquote}{Tom Weston} When considering the adeles and ideles, it is their topology as much as their
algebraic structure that is of interest. Many important results in number theory
translate into simple statements about the topologies of the adeles and ideles. For
example, the finiteness of the ideal class group and the Dirichlet unit theorem are
equivalent to a certain quotient of the ideles being compact and discrete. 
\end{aquote}
\vspace{5pt}
\begin{aquote}{Serge Lang} In classical number theory, one embeds a number field in the Cartesian
product of its completions at the archimedean absolute values, i.e. in a
Euclidean space. In more recent years (more precisely since Chevalley
introduced ideles in 1936, and Weil gave his adelic proof of the RiemannRoch
theorem soon afterwards), it has been found most convenient to
take the product over· the completions at all absolute values, including
the p-adic ones, with a suitable restriction on the components.
\end{aquote}
\vspace{5pt}
\begin{aquote}{Neukirch}
The role held in local class field theory by the multiplicative group of the
base field is taken in global class field theory by the idele class group. The
notion of idele is a modification of the notion of ideal. It was introduced
by the French mathematician Claude Chevalley (1909-1984) with a view
to providing a suitable basis for the important local-to-global principle, i.e.,
for the principle which reduces problems concerning a number field $K$ to
analogous problems for the various completions $K_\mfk{p}$. Chevalley used the term \emph{ideal element}, which was abbreviated as id. el.
\end{aquote}
\newpage
\printindex

\end{document}